\documentclass[10pt, reqno]{amsart}

\usepackage{amsmath,amssymb,amsthm,amsfonts,verbatim}
\usepackage{microtype}
\usepackage{graphicx}
\usepackage{hyperref}
\usepackage{mathrsfs}
\usepackage{color}
\usepackage{enumerate}
\usepackage{cite}
\usepackage{subcaption}
\usepackage{svg}
\usepackage{enumitem}

\usepackage[margin=1.5 in]{geometry}

\theoremstyle{plain}
\newtheorem{theorem}{Theorem}[section]
\newtheorem{maintheorem}{Theorem}

\newtheorem{question}[theorem]{Question}

\newtheorem{proposition}[theorem]{Proposition}
\newtheorem{lemma}[theorem]{Lemma}

\newtheorem{fact}[theorem]{Fact}
\newtheorem{corollary}[theorem]{Corollary}
\newtheorem{claim}[theorem]{Claim}

\theoremstyle{definition}
\newtheorem{definition}[theorem]{Definition}
\newtheorem{example}[theorem]{Example}

\newtheorem{remark}[theorem]{Remark}

\newcommand{\ZZ}{\mathbb{Z}}
\newcommand{\RR}{\mathbb{R}}

\newcommand{\CC}{\mathbb{C}}

\newcommand{\cH}{\mathcal{H}}

\DeclareMathOperator{\Mod}{Mod}

\DeclareMathOperator{\FMod}{FMod}
\DeclareMathOperator{\Stab}{Stab}
\DeclareMathOperator{\Cadm}{\mathscr{C}_{adm}}
\newcommand{\CS}{\mathscr{C}(S)}
\newcommand{\sing}{\underline{\kappa}}

\DeclareMathOperator{\diam}{diam}
\newcommand{\genus}{g}
\DeclareMathOperator{\Arf}{Arf}

\newcommand{\cK}{\mathcal{K}}
\newcommand{\ocK}{\overline{\mathcal{K}}}
\newcommand{\sC}{\mathscr{C}}
\newcommand{\fS}{\mathfrak{S}}
\newcommand{\ofS}{\overline{\mathfrak{S}}}
\newcommand{\cD}{\mathcal{D}}
\newcommand{\cG}{\mathcal{G}}
\newcommand{\markH}{{\mathcal{H}}_{\phi}}
\newcommand{\barmarkH}{\overline{\mathcal{H}_{\phi}}}
\newcommand{\BATS}{\overline{\mathcal{T}_{g,n}}}
\newcommand{\DMcomp}{\overline{\mathcal{M}_{g,n}}}

\newcommand{\CRV}{\mathscr{E}(\barmarkH)}
\newcommand{\CRE}{\mathscr{D}(\barmarkH)}
\newcommand{\scrC}{\mathscr{C}}

\newcommand{\cT}{\mathcal T}
\newcommand{\cM}{\mathcal M}
\newcommand{\cX}{\mathcal X}

\newcommand{\tsh}[1]{\left\{\kern-.7ex\left\{#1\right\}\kern-.7ex\right\}}

\newcommand{\wit}{\operatorname{Wit}}

\title[Admissible curve graphs and the boundary of strata]{Hierarchical hyperbolicity of admissible curve graphs and the boundary of marked strata}

\author{Aaron Calderon and Jacob Russell}

\begin{document}
\begin{abstract}
We show that for any surface of genus at least 3 equipped with any choice of framing, the graph of non-separating curves with winding number 0 with respect to the framing is hierarchically hyperbolic but not Gromov hyperbolic.
We also describe how to build analogues of the curve graph for marked strata of abelian differentials that capture the combinatorics of their boundaries, analogous to how the curve graph captures the combinatorics of the augmented Teichm{\"u}ller space. These curve graph analogues are also shown to be hierarchically, but not Gromov, hyperbolic.
\end{abstract}

\maketitle

\thispagestyle{empty}
\vspace{-2em}

\section{Introduction}

The moduli space $\Omega\cM_g$ of genus $g$ Abelian differentials forms a bundle over the usual moduli space $\cM_g$ of genus $g$ Riemann surfaces.
This bundle decomposes into {\em strata},
subvarieties which parametrize differentials with a given number and order of zeros and which are the ambient theatre for Teichm{\"u}ller dynamics.
The overall structure of strata is still poorly understood, and recent work has been largely guided by the following:

\begin{question}\label{Q:stratavsMg}
How similar are strata and $\cM_g$?
\end{question}

There has been a great deal of success constructing compactifications of strata akin to the Deligne--Mumford compactification of $\cM_g$ \cite{EMZ,IVC, multiscale}. The structure of these boundaries can then be used to compute constants of dynamical interest \cite{EMZ}, perform intersection theory on strata
\cite{CMSZ_int}, and compute their Euler characteristics \cite{CMZEuler}, among many other things.

Another version of Question \ref{Q:stratavsMg} deals with their fundamental groups.
Recall that $\cM_g$ is an (orbifold) $K(\pi, 1)$ for the usual mapping class group $\Mod(S)$, the group of homeomorphisms of the surface up to homotopy.
By analogy, Kontsevich predicted that each connected component of a stratum should be a $K(\pi, 1)$ for ``some mapping class group'' \cite{KZ_strings}.
In \cite{CS}, the first author and Salter showed that the fundamental groups of strata are closely related to {\em framed mapping class groups} $\FMod(S,\phi)$, the stabilizers inside $\Mod(S)$ of trivializations $\phi: TS \cong S \times \RR^2$ (see \S\ref{sec:background} for a formal definition).
Apisa, Bainbridge, and Wang subsequently showed that certain strata of twisted 1-forms are $K(\pi,1)$'s for framed mapping class groups \cite{ABW}. A group-theoretic analogue of Question \ref{Q:stratavsMg} is thus:

\begin{question}\label{Q:FModvsMod}
How similar are $\FMod(S,\phi)$ and $\Mod(S)$?
\end{question}

\subsection{Curve graphs and strata}
This paper initiates the study of Questions \ref{Q:stratavsMg} and \ref{Q:FModvsMod} from the coarse-geometric perspective by analyzing the geometry of certain curve graphs.

The classical curve graph $\CS$ has a vertex for each isotopy class of essential simple closed curve on an orientable surface and an edge when two curves can be realized disjointly \cite{Harvey_CS}.
In addition to this topological interpretation, this graph also plays the role of (the 1-skeleton of) a Tits building for Teichm{\"u}ller space $\cT_g$, recording the incidences of top-dimensional boundary strata of the {\em augmented Teichm{\"u}ller space}, a certain bordification of $\cT_g$ that ``lifts'' the Deligne-- Mumford compactification of $\cM_g$ (see \S\ref{subsec:C(S) as nerve}).

Masur and Minsky famously proved that $\CS$ is Gromov hyperbolic \cite{MM1}. This marquee result has far-reaching implications for the coarse geometry of the mapping class group \cite{Ivanov_curve_graph,MM2}, the geometry of Teichm\"uller space \cite{MM1, Rafi_short_curves}, and the structure of hyperbolic 3-manifolds \cite{Minsky_ELC1,BCM_ELC2}.
More generally, the geometry of curve graphs has proven useful in a variety of settings; examples of this paradigm include relationships between the pants graph/the Weil--Petersson metric on $\cT_g$ \cite{Brock_WP_volumes,BF_WP_rank},
the Torelli complex and separating curve graph/the Torelli subgroup and the Johnson kernel \cite{FI_Torelli,BM_johnson_kernel}, and the disk graph/the handlebody group and Heegaard splittings \cite{Hensel_handlebody, MS_disk_graph}.
\medskip

As a first step towards Question \ref{Q:FModvsMod}, we study a topological analogue of $\CS$ that takes the framing into account.
Any framing $\phi: TS \cong S \times \RR^2$ can be used to measure the winding number of a smooth, oriented curve in $S$ by lifting the curve to $TS$ via its tangent vector, projecting to the second coordinate, then measuring the winding number of the image about $0 \in \RR^2$.
A simple closed curve on $S$ is {\em admissible} for $\phi$ if it is nonseparating and has zero winding number, and the {\em admissible curve graph} $\Cadm(S, \phi)$ is the subgraph of $\CS$ spanned by admissible curves.

The framed mapping class group $\FMod(S,\phi)$ preserves the winding number of every curve, hence acts on $\CS$ with infinitely many orbits of vertices.
In contrast, $\FMod(S,\phi)$ acts on $\Cadm(S, \phi)$ with finitely many orbits of vertices and edges (Proposition \ref{prop:transitiveadm}), indicating that the admissible curve graph is better adapted to study $\FMod(S,\phi)$.

Our first main result is that the admissible curve graph is {\em not} Gromov hyperbolic, but does possess a generalized notion of hyperbolicity.

\begin{maintheorem}\label{mainthm:CadmHHS}
	For any surface $S = S_{g,n}$ of genus $g \ge 3$ and any framing $\phi$ of $S$, the admissible curve graph $\Cadm(S, \phi)$ is hierarchically hyperbolic (but not Gromov hyperbolic).
\end{maintheorem}

\emph{Hierarchical hyperbolicity} was introduced by Behrstock, Hagen, and Sisto to unify similarities between the coarse geometry of mapping class groups, Teichm\"uller spaces, and right-angled Artin groups \cite{BHS_HHSI}. Briefly, this framework allows one to understand  the geometry of a space by projecting it onto a collection of Gromov hyperbolic spaces.
The presence of ``orthogonal'' projections leads to quasi-isometrically embedded flats, hence a failure of Gromov hyperbolicity.
\medskip

We can also define a geometric analogue of $\CS$ that captures the intersection pattern of the boundary of a marked stratum.
More precisely, since holomorphic differentials are determined up to scaling by the order and position of their zeros, any stratum component $\cH \subset \Omega^1\cM_g$ is an (orbifold) $\CC^*$-bundle over a subvariety of $\cM_{g,n}$, the moduli space of genus $g$ Riemann surfaces with $n$ marked points.
Let us conflate $\cH$ with this subvariety.

Take any non-hyperelliptic stratum component $\cH \subset \cM_{g,n}$, let $\cH_{\phi}$ be any component of the preimage of $\cH$ in $\cT_{g,n}$, and consider its closure $\barmarkH$ in the augmented Teichm{\"u}ller space $\BATS$.
Define a graph $\mathscr{C}(\barmarkH)$ whose vertices are those multicurves $\gamma$ such that $\barmarkH \cap \cT_{g,n}(\gamma) \neq \emptyset$, where $\cT_{g,n}(\gamma)$ is the boundary stratum of $\BATS$ in which $\gamma$ is pinched, and whose edges are given by inclusion.
The intricate structure of the boundary of $\barmarkH$ means there are other natural ways to define this graph (see Sections \ref{subsec:multiscale} and \ref{subsec:divisorial}), but they all turn out to be quasi-isometric to $\mathscr{C}(\barmarkH)$.

The geometry of $\mathscr{C}(\barmarkH)$ is closely linked to that of $\Cadm(S,\phi)$, and using Theorem \ref{mainthm:CadmHHS} plus structural results about compactifications of strata \cite{IVC, multiscale}, we prove:

\begin{maintheorem}\label{mainthm:bdrycx}
	For any non-hyperelliptic stratum component $\cH \subset \Omega^1\cM_g$ with $g \ge 5$, the graph $\mathscr{C}(\barmarkH)$ is hierarchically hyperbolic (but not Gromov hyperbolic).
\end{maintheorem}

\begin{remark}
As shown in \cite[Corollary 1.2]{CS}, admissible curves are exactly the core curves of cylinders on surfaces in $\markH$.
One can also construct a partial bordification of $\markH$ in which only cylinders are allowed to degenerate; the combinatorics of how this space meets $\partial \BATS$ then correspond to $\Cadm(S, \phi)$.
Thus Theorem \ref{mainthm:CadmHHS} can also be interpreted as a statement about the coarse geometry of $\markH$.
\end{remark}

\begin{remark}
Our restriction to non-hyperelliptic components is because the hyperelliptic ones do not exhibit new phenomena.
Indeed, hyperelliptic stratum components are essentially strata of quadratic differentials on $\mathbb{CP}^1$, which are in turn parametrized by their poles and zeros. 
Thus we can understand compactifications of hyperelliptic stratum components entirely in terms of the Deligne--Mumford compactification of $\cM_{0,n}$.
\end{remark}

\begin{remark}
The restriction to $g \ge 3$ in Theorem \ref{mainthm:CadmHHS} is because for $g=1,2$ the admissible curve graph is not necessarily connected.
The restriction to $g \ge 5$ in Theorem \ref{mainthm:bdrycx} comes from the fact that the main theorem of \cite{CS} relating $\pi_1(\cH)$ and $\FMod(S,\phi)$ only applies for $g \ge 5$.
In Section \ref{sec:bdrycx} we give a (slightly circuitous) definition of $\mathscr{C}(\barmarkH)$ that agrees with the one given above for $g \ge 5$ and for which Theorem \ref{mainthm:bdrycx} holds in genus 3 and 4.
In particular, all of the proofs in this paper hold for $g \ge 3$.
\end{remark}

Curve graph techniques have been used successfully to study certain $\mathsf{GL}_2\RR$--invariant subvarieties of $\Omega^1\cM_g$:
\cite{Tang} proved that Veech groups are undistorted in $\Mod(S)$, \cite{RScovers} proved a similar result for covering constructions, and \cite{AHW_geo} used curve graphs to study the geometry of totally geodesic subvarieties of Teichm{\"u}ller space.
It is our hope that the tools developed in this paper will yield insights into both the intrinsic and extrinsic geometry of framed mapping class groups and strata. For example, we ask:

\begin{question}
Is $\FMod(S,\phi)$ distorted in $\Mod(S)$? Are strata distorted in $\cM_{g,n}$?
\end{question}

\subsection{Outline of proof and paper}
To prove Theorems \ref{mainthm:CadmHHS} and \ref{mainthm:bdrycx}, we need to exhibit projections from $\Cadm(S,\phi)$ and $\mathscr{C}(\barmarkH)$ to Gromov hyperbolic spaces.
In both settings, we use Masur and Minsky's subsurface projection maps to the curve graphs of \emph{witnesses} --- subsurfaces of $S$ that intersect every admissible curve. 
This approach was inspired by work of Vokes, who showed that a wide variety of graphs of curves are hierarchically hyperbolic using their subsurface projection maps to witnesses \cite{Vokes_HHS}. Vokes first uses the set of witnesses to build a hierarchically hyperbolic ``model graph'' $\cK$, then shows that if the graph of curves admits a cobounded action of $\Mod(S)$ then it is quasi-isometric to $\cK$.

To prove Theorem \ref{mainthm:CadmHHS}, we construct a hierarchically hyperbolic model $\cK$ for $\Cadm(S, \phi)$ {\`a} la Vokes (Section \ref{sec:models}). However, we cannot employ her quasi-isometry as $\Cadm(S, \phi)$ does not admit an action by all of $\Mod(S)$ and the action of $\FMod(S, \phi)$ on $\cK$ is not sufficiently cofinite to adapt her argument. Instead, we  construct a novel quasi-isometry $\cK \to \Cadm(S, \phi)$ via the graph $\cG$ of genus separating curves (Section \ref{sec:Cadm_qi_model}). 
The graph $\cG$ can be quasi-isometrically realized as a ``blow-up'' of $\cK$, while $\Cadm(S,\phi)$ is quasi-isometric to a ``cone-off'' of $\cG$.
To build the map $\cK \to \Cadm(S, \phi)$, we show that the blown-up subsets from $\cK \to \cG$ coarsely match the coned-off subsets from $\cG \to \Cadm(S,\phi)$.
This step requires some fairly delicate computations with curves on surfaces.

Theorem \ref{mainthm:bdrycx} follows by constructing a quasi-isometric model for $\mathscr{C}(\barmarkH)$ entirely in terms of framing data.
This requires unpacking some of the finer structure of the boundary, as developed in \cite{multiscale}, and giving topological interpretations to many of the objects involved. These steps are accomplished in Section \ref{sec:bdrycx}.
In this section, we also build a trio of graphs whose definitions interpolate between the structure of $\partial \barmarkH$ and framing data.

In the final Section \ref{sec:transgeom}, we show that the three graphs from Section \ref{sec:bdrycx} are all quasi-isometric, and that they are quasi-isometric to a Vokes model graph $\overline{\cK}$.
Again, there is not sufficient transitivity to apply Vokes's methods, and the construction of a quasi-isometry is quite subtle.
The graph $\overline{\cK}$ is an $\FMod(S,\phi)$--equivariant cone-off of the model $\cK$ for $\Cadm(S, \phi)$, and
the inclusion $\Cadm(S, \phi) \hookrightarrow \mathscr{C}(\barmarkH)$ is also an equivariant cone-off. As in the case of Theorem \ref{mainthm:CadmHHS}, the main difficulty is then showing that these two cone-offs coarsely match.

A common theme running throughout this paper is that if one understands the $\FMod(S, \phi)$ action on configurations of curves and subsurfaces well enough, then many surface-topological arguments can be adapted to the framed setting with a little extra care and effort.
As such, we prove a number of transitivity results (Propositions \ref{prop:transitiveadm}, \ref{prop:2 level trans}, and \ref{prop:realize_adm}) for the $\FMod(S, \phi)$ action that may be of broader interest.

\subsection*{Acknowledgments}
We thank Martin M{\"o}ller and Frederik Benirschke for helping us to understand the structure of the boundary of strata. We also thank Dan Margalit for helpful comments on an early draft of this paper.

This paper grew out of conversations begun in Spring 2022 when JR was attending the trimester program ``Groups acting on fractals, hyperbolicity and self-similarity.'' We gratefully acknowledge financial support from the Institut Henri Poincar{\'e} (UAR 839 CNRS-Sorbonne Universit{\'e}) and LabEx CARMIN (ANR-10-LABX-59-01),
and thank them for their hospitality.
AC acknowledges support from NSF grants DMS-2005328 and DMS-2202703.
JR acknowledges support from NSF grant DMS-2103191.

\section{Surfaces, curves, and framings}\label{sec:background}
Let us first recall some basic surface-topological notions and set our notation for the rest of the paper.
Let $S = S_{g,n}$ denote an orientable surface with genus $g$ and $n$ punctures.
The \emph{complexity} of $S=S_{g,n}$ is $\xi(S) = 3g-3 +n$.  By a \emph{curve} on $S$ we mean an isotopy class of an essential (i.e., non-nulhomotopic), non-peripheral (i.e., not homotopic to a puncture), simple closed curve on $S$.
An \emph{arc} on $S$ is an isotopy class of essential, non-peripheral simple arcs running between the punctures.
Curves and arcs are unoriented unless we say otherwise.
By a \emph{subsurface} of $S$, we mean an isotopy class of an essential, non-peripheral, (relatively) closed subsurface of $S$. For two subsurfaces $U$ and $V$, we say $U\subseteq V$ if $U$ and $V$ can be realized such that $U$ is contained in $V$.
We say two curves and/or subsurfaces are \emph{disjoint} if their isotopy classes can be realized disjointly. Otherwise, we say they \emph{intersect}. A \emph{multicurve} on $S$ is a collection of distinct, disjoint curves on $S$.
Throughout the paper, we use lowercase Latin letters to refer to curves, Greek letters to multicurves and arcs, and uppercase letters to subsurfaces.

Given two multicurves $\alpha$, $\beta$ on $S$, we let $i(\alpha,\beta)$ denote their \emph{geometric intersection number}. If $\alpha$ and $\beta$ are oriented curves, then $\langle \alpha, \beta \rangle$ will denote their \emph{algebraic intersection number}.
If a multicurve $\alpha$ intersects a subsurface $W \subseteq S$, then $\alpha \cap W$ is the isotopy class (relative to $\partial W$) of curves and arcs obtained by taking the intersection of $W$ with a representative for $\alpha$ that realizes $i(\alpha, \partial W)$. Two arcs $\alpha_1,\alpha_2$ on the subsurface $W$ are \emph{parallel} if they are isotopic by isotopies fixing $\partial W$ setwise but not pointwise. 

If $\alpha$ is a multicurve on $S$, then $S \setminus \alpha$ will denote the closed subsurface obtained by removing a small open neighborhood of each curve in $\alpha$ from $S$. Similarly, if $W$ is a subsurface of $S$, then $S \setminus W$ is the closed subsurface obtained by removing a small open neighborhood of $W$ from $S$. We denote the genus of a subsurface $W \subseteq S$ by $\genus(W)$. 

The mapping class group, $\Mod(S)$, is the group of homeomorphisms of $S$ that fix each of its punctures, modulo isotopy.
The mapping class group is generated by {\em Dehn twists}: for any simple closed curve $c$, let $T_c$ denote the homeomorphism obtained by cutting open $S$ along $c$, twisting one of the boundary components of $S \setminus c$ once to the left, and then regluing.

\subsection{Framings and winding numbers}
A {\em framing} of a surface $S$ is a trivialization of its tangent bundle $\phi: TS \xrightarrow{\sim} S \times \mathbb{R}^2$. For surfaces of genus not equal to 1, the existence of a framing requires $S$ to have punctures and/or boundary.
Throughout this paper we will think of $S$ as having punctures.

We are interested in the set of framings up to isotopy; these were called ``absolute framings'' in \cite{CS}.
Isotopy classes of framings can be described by the discrete invariant of a ``winding number function'' as follows.
Given any $C^1$ immersed curve $\gamma: [0,1] \to S$, the tangent framing $(\gamma, \gamma')$ gives a curve in $TS \cong S \times \mathbb{R}^2$.
Projecting into the second factor gives a loop in $\mathbb{R}^2 \setminus \{0\}$ and so one can measure the {\em winding number} $\phi(\gamma)$ of $\gamma'$ about $0$.
This number is an invariant of the isotopy class of framing as well as the isotopy class of $\gamma$ (though not its homotopy class), and so to every framing $\phi$ we have an associated winding number function of the same name
\[\phi: \mathcal{S} \to \mathbb{Z},\]
where $\mathcal{S}$ denotes the set of isotopy classes of oriented simple closed curves.
It is not hard to show that the function $\phi$ is actually a complete invariant of the isotopy class of the framing \cite[Proposition 2.4]{RW}, and so for the remainder of the paper we will conflate a(n isotopy class of) framing and its associated winding number function.

\begin{remark}
In a previous version of this paper, we considered surfaces with boundary where the framing was allowed to vary on the boundary. This is equivalent to the absolute framings we now consider; see \cite[Section 6.2]{CS}.
\end{remark}

Winding number functions have two very important properties, which were first elucidated by Humphries and Johnson \cite{HJ}. As a consequence, a framing is completely determined (up to isotopy) by its values on a basis for homology. 

\begin{lemma}[Humphries--Johnson]\label{lemma:HJ}
	Any winding number function $\phi$ associated to a framing satisfies the following properties.
	\begin{enumerate}
		\item \label{item:TL} (Twist-linearity) Let $a, b \subset S$ be oriented simple closed curves. Then
		\[
		\phi(T_a(b)) = \phi(b) + \langle b,a \rangle \phi(a),
		\]
		where $\langle \cdot, \cdot \rangle: H_1(S;\ZZ) \times H_1(S; \ZZ) \to \ZZ$ denotes the algebraic intersection pairing.
		\item \label{item:HC} (Homological coherence) Let $U \subset S$ be a subsurface and let $c_1, \dots, c_k$ denote its boundary components and the peripheral loops about its punctures, oriented such that $U$ lies to the left of each $c_i$. Then
		\[
		\sum_{i = 1}^k \phi(c_i) = \chi(U),
		\]
		where $\chi(U)$ denotes the Euler characteristic.
	\end{enumerate}
\end{lemma}

Let $\Delta_1, \ldots, \Delta_k$ denote small loops about the punctures of $S$ (oriented with the surface on their left); then the {\em signature} of a framing $\phi$ is the tuple 
\[\text{sig}(\phi) := (\phi(\Delta_1), \ldots, \phi(\Delta_k)) \in \mathbb{Z}^k.\]
A framing is said to be of {\em holomorphic type} if every $\phi(\Delta_i)$ is negative; this terminology comes from the fact that the horizontal vector fields of holomorphic abelian differentials give rise to such framings (compare Section \ref{subsec:stratabasics}).

\begin{remark}\label{rmk:holtype vs holdiff}
We note that {\em not} every framing of holomorphic type comes from a holomorphic abelian differential. This is the case for framings on surfaces of genus at least 3, but the following families of framings do not come from abelian differentials due to certain low-complexity strata being empty (see just below for the definitions of $\Arf_1$ and $\Arf$).
\begin{itemize}
\item $g=1$, $b=1$, and $\Arf_1(\phi) \neq 0$.
\item $g=2$, $b=1$, and $\Arf(\phi) = 0$.
\end{itemize}
\end{remark}

The peripheral curves $\Delta_i$ span a $k-1$ dimensional subspace of $H_1(S)$, so we can construct all framings with a given signature by specifying the values on $2g$ homologically independent curves \cite[Remark 2.7]{CS}. One particularly nice configuration is as follows:

\begin{definition}
A collection of simple closed curves $\mathcal B = \{a_1, b_1, \ldots, a_g, b_g\}$ on $S$ is called a {\em geometric symplectic basis} (GSB) if 
$i(a_i, b_i) = 1$ for all $i$ and all other pairs of curves from $\mathcal B$ are disjoint.
\end{definition}

\subsection{Framed mapping class groups}
The {\em framed mapping class group} $\FMod(S, \phi)$ associated to a framing $\phi$ is the stabilizer of $\phi$ in $\Mod(S)$ up to isotopy.
Equivalently, and more usefully, $f \in \FMod(S, \phi)$ if and only if it preserves all winding numbers, i.e., 
\[(f\cdot \phi)(a) := \phi(f^{-1}(a)) = \phi(a)\]
for every $a \in \mathcal S$.
In light of Lemma \ref{lemma:HJ}, in order to check if an element $f \in \Mod(S)$ actually preserves $\phi$, it suffices to show that show that $f$ preserves the $\phi$--winding numbers of all curves of a GSB.

Throughout the paper, a particularly important role will be played by the set of non-separating simple closed curves with $\phi(a) = 0$ (note that this does not depend on orientation); these curves are said to be {\em admissible}.
By twist-linearity (Lemma \ref{lemma:HJ}.\ref{item:TL}),  Dehn twists in admissible curves are always in $\FMod(S, \phi)$, and in \cite{CS} it is shown (for $g \ge 5$) that $\FMod(S, \phi)$ is generated up to finite index by admissible twists.

Since each orbit of $\Mod(S)$ on the set of framings has infinite size (this is an immediate consequence of Lemma \ref{lemma:HJ}) and $\FMod(S, \phi)$ is a stabilizer, it is an infinite-index subgroup.
Along the same lines, understanding the possible conjugacy classes of $\FMod(S, \phi)$ for different $\phi$ is equivalent to listing the $\Mod(S)$ orbits.
To state this ``classification of framed surfaces'' \cite{Kawazumi} (see also \cite{RW} for the relatively framed version), we first need to recall the definitions of the Arf invariant and its genus 1 version; see \cite[\S 2.2]{CS}, \cite[\S 2.4]{Kawazumi}, and \cite[\S 2.4]{RW} for more detailed discussions.

Suppose first that $g=g(S)\ge 2$ and that every $\phi(\Delta_i)$ is odd. In this case, we say that $\phi$ is of {\em spin type}.
\footnote{In this case, the framing induces a (2-)spin structure on the closed surface obtained by capping off all boundary components, and the Arf invariant of the framing coincides with the parity of the spin structure.}
Fix a geometric symplectic basis $\{a_1,b_1,\dots,a_g,b_g\}$ on $S$. Then the {\em Arf invariant} of $\phi$ is defined to be
\begin{equation}\label{eqn:arfdef}
\Arf(\phi) := \sum_{i=1}^g \left( \phi(a_i) +1 \right) \left( \phi(b_i) +1\right)  \mod 2.
\end{equation}
This invariant turns out to only be well-defined when each $\phi(\Delta_i)$ is odd, and in this setting it does not depend on our choice of GSB.
If $g=1$, then there is an $\mathbb{Z}$-valued refinement of the Arf invariant which we denote by
\[\Arf_1(\phi):= \gcd( \phi(c), \phi(\Delta_1) + 1, \ldots, \phi(\Delta_k)+1 \mid 
c \text{ is a non-separating simple closed curve}).\]

\begin{theorem}\label{thm:classframed}
Two framings $\phi$ and $\phi'$ of $S$ are in the same $\Mod(S)$ orbit if and only if
\begin{itemize}
    \itemindent=12pt
    \item [$(g=0)$] $\text{sig}(\phi) = \text{sig}(\phi')$
    \item [$(g=1)$] $\text{sig}(\phi) = \text{sig}(\phi')$ and $\Arf_1(\phi) = \Arf_1(\phi')$
    \item [$(g\ge2)$] $\text{sig}(\phi) = \text{sig}(\phi')$ and if $\phi$ and $\phi'$ are of spin type, then $\Arf(\phi) = \Arf(\phi')$.
\end{itemize}
\end{theorem}

In particular, for genus at least 2 there are only ever at most $2$ distinct conjugacy classes of framed mapping class groups.

The Arf invariant interacts in a complicated way with taking subsurfaces $V \subset S$; sometimes the Arf invariant of $\phi|_V$ is forced by the topology of $V$, and sometimes it can vary for different $V$ and $V'$ of the same topological type.
For later use, we record an example of this phenomenon below. See also the proofs of Propositions \ref{prop:transitiveadm} and \ref{prop:2 level trans}.

\begin{lemma}\label{lem:fullgenusArf}
Suppose that $V \subset S$ is a connected subsurface of full genus.
\begin{enumerate}
\item If $g(S) \ge 2$ and $\phi$ is of spin type, then $\Arf(\phi) = \Arf(\phi_V)$.
\item $g(S) = 1$ and $\phi$ is of holomorphic type, then $\Arf_1(\phi) = \Arf_1(\phi|_V)$
\end{enumerate}
\end{lemma}
\begin{proof}
When $S$ has genus at least 2, this is an immediate consequence of \eqref{eqn:arfdef}. In the case when $S$ has genus 1, homological coherence together with holomorphic type imply that two curves which differ by sliding over a boundary component must have the same winding number. Thus for any simple closed curve $c$ on $S$, there is some $c' \subset V$ with $\phi(c) = \phi(c')$, and hence their genus-1 Arf invariants must agree.
\end{proof}

Note that statement (2) is false if one does not assume holomorphic type.

\subsection{Framed change-of-coordinates}
The standard change-of-coordinates principle for the entire mapping class group roughly states that given two multicurves $\gamma$ and $\delta$, there is some $f \in \Mod(S)$ taking $\gamma$ to $\delta$ if and only if $S \setminus \gamma$ and $S \setminus \delta$ have the same topological type and are glued together in the same way.
This technique is often used in surface topology to show the existence of certain configurations of curves with prescribed intersection pattern and to show the transitivity of the $\Mod(S)$ action on such configurations.
Its proof is a corollary of the classification of surfaces: one uses the classification to build a homeomorphism between the complements then extends that to a self-homeomorphism of $S$.

In the framed setting, we can similarly use Theorem \ref{thm:classframed} to show the existence of configurations with certain intersection pattern and winding number (compare \cite[Proposition 2.5]{CS}). For example, we can quickly show that (sub)surfaces with genus always contain admissible curves. Essentially the same statement appears as Corollary 4.3 of \cite{salter_higher_spin}, but we include a proof as we will repeatedly use this statement throughout the paper.

\begin{lemma}\label{lem:genusadm}
For any framing $\phi$ on a surface $S$ of positive genus, there is some non-separating simple closed curve $a \subset S$ with $\phi(a)=0$.
\end{lemma}
\begin{proof}
Fix a GSB $\{a_1, \ldots, b_g\}$ on $S$. Then by stipulating winding numbers on our GSB
we can build a framing $\psi$ such that 
\begin{itemize}
    \item $\text{sig}(\phi) = \text{sig}(\psi)$
    \item $\psi(a_1) = 0$, and 
    \item if $g(S)=1$ then $\Arf_1(\psi) = \Arf_1(\phi)$, or
    \item if $g(S) \ge 2$ and $\phi$ is of spin type then $\Arf(\psi) = \Arf(\phi)$.
\end{itemize} 
Now by Theorem \ref{thm:classframed} there is some homeomorphism $f \in \Mod(S)$ taking $\psi$ to $\phi$, and the curve $f(a_1)$ is our desired admissible curve.
\end{proof}

Along the same lines, one can show that $S$ always admits a GSB with given winding numbers so long as those winding numbers yield the correct Arf invariant; the proof is left to the reader. See also the proof of the first part of \cite[Proposition 2.15]{CS}.

\begin{lemma}\label{lem:existsGSB}
Let $\phi$ be a framing of a surface $S$ of genus $g \ge 1$ and fix any tuple of integers $(x_1, y_1, \ldots, x_g, y_g)$ such that
\begin{itemize}
    \item if $g=1$, then 
    $\gcd(x_1, y_1, \phi(\Delta_1)+1, \ldots, \phi(\Delta_n)+1) = \Arf_1(\phi),$
    \item if $g \ge 2$ and $\phi$ is of spin type, then
    \[ \sum_{i=1}^g (x_i+1) (y_i +1) = \Arf(\phi) \mod 2\]
    \item if $g \ge 2$ and $\phi$ is not of spin type, then we impose no conditions on the tuple.
\end{itemize}
Then there is a GSB $\mathcal B = \{ a_1, b_1, \ldots, a_g, b_g\}$ on $S$ such that $\phi(a_i) = x_i$ and $\phi(b_i)=y_i$.
\end{lemma}

In particular, any surface of genus at least $2$ contains nonseparating curves of arbitrary winding number.

The classification of framed surfaces can also be used to easily obstruct transitivity of the $\FMod(S, \phi)$ action.
For example, $\FMod(S,\phi)$ does not act transitively on the set of curves that separate off a genus 1 subsurface with one boundary component, even though those curves all have the same topological type and same winding number.
The reason is that the induced framing on the subsurface may have different $\Arf_1$ invariant.

We caution the reader that Theorem \ref{thm:classframed} does not imply transitivity on the set of multicurves of the same topological type that induce homeomorphic framings on each subsurface.
Indeed, suppose that some $\phi(\Delta_i)$ is even so $\phi$ does not have an induced Arf invariant.
If we consider the set of multicurves $\gamma = c \cup d$ where $c$ cuts off a genus 1 subsurface with one puncture and $d$ is an admissible curve on that subsurface, then the paragraph above implies that $\FMod(S, \phi)$ does not act transitively on this set, even though there is only one $\Mod(S \setminus \gamma)$ orbit of framing on $S \setminus \gamma$.
At issue is what happens when we try to glue together framings on subsurfaces to a framing on the entire surface; this can be dealt with by using {\em relative} framings and being careful about boundary conditions (compare the proof of Lemma 5.3 in \cite{CS}).
Since such arguments require a fair amount of delicacy and are beyond what we need in this paper, we will restrict ourselves to proving those transitivity results we will need in the sequel.

\begin{proposition}\label{prop:transitiveadm}
Let $\phi$ be a framing of a surface $S$ of genus at least 3. Then $\FMod(S, \phi)$ acts transitively on the set of pairs of admissible curves of the same topological type.
That is, if $\gamma, \gamma'$ are pairs of admissible curves and there is some $g \in \Mod(S)$ taking $\gamma$ to $\gamma'$, then there is also some $f \in \FMod(S, \phi)$ taking $\gamma$ to $\gamma'$.
\end{proposition}

Before proving Proposition \ref{prop:transitiveadm}, we first record a useful lemma that allows us to adjust the winding numbers of curves in a configuration without changing their intersection properties. A similar statement appears as Corollary 4.4 of \cite{salter_higher_spin}.

\begin{lemma}\label{lem:adjustwinding}
Let $\phi$ be a framing of a surface $S$ and let $c_1, \ldots, c_k,d$ be a collection of simple closed curves. Assume there is some subsurface $T \subset S$, disjoint from all of the listed curves, such that either
\begin{itemize}
    \item $g(T) \ge 2$, or
    \item $g(T)=1$ and $\Arf_1(\phi|_T)=1$.
\end{itemize}
Suppose also that there is some arc $\varepsilon$ connecting $d$ to $T$ that is disjoint from all $c_i$.
Then for any $z \in \mathbb{Z}$, there is a simple closed curve $d_z$ such that $\phi(d_z)=z$ and $i(c_i, d_z) = i(c_i,d)$ for all $i$.
\end{lemma}
\begin{proof}
Orient $d$ such that the arc from $d$ to $T$ exits $d$ from its left-hand side.

Suppose first that $g(T)=2.$
Then by Lemma \ref{lem:existsGSB} there is a nonseparating curve $e$ on $T$ with winding number $-z - \phi(d) -1$.
Since $d$ is not separated from $T$, we may concatenate $\varepsilon$ with an arc connecting $\partial T$ to the left side of $e$ and take the connect sum of $d$ and $e$ along this composite arc. Let $d_z$ be the resulting curve; then by homological coherence (Lemma \ref{lemma:HJ}.\ref{item:HC}) we have that 
\[\phi(d_z) + \phi(d) + \phi(e) = -1\]
and so $d_z$ is our desired curve.
It clearly has the same intersection pattern as $d$ with each $c_i$ since we have only altered $d$ away from $c_i$ (see also the proof of \cite[Corollary 4.4]{salter_higher_spin}).

In the case that $g(T) = 1$, our assumption on $\Arf_1(\phi|_T)$ implies (via Lemma \ref{lem:existsGSB}) that there is some GSB $(a,b)$ on $T$ with $\phi(a)=1$.
Choose an arc from $\partial T$ to $b$ disjoint from $a$, then take the connected sum of $d$ with $b$ along the concatenation of $\varepsilon$ with this arc. This results in a new curve $d'$ that has the same intersection pattern as $d$ with each $c_i$ and meets $a$ exactly once.
Twist-linearity (Lemma \ref{lemma:HJ}.\ref{item:TL}) now implies that by twisting around $a$ we can alter the winding number of $d'$ by an arbitrary amount to find our desired $d_z$.
\end{proof}

One particularly important consequence is that we can complete any admissible curve to a partial GSB while specifying the winding number of the transverse curve.

\begin{corollary}\label{cor:specifytransverse}
For any surface of genus at least $2$, any admissible $a$, and any $z \in \mathbb{Z}$, there is a curve $b$ with $i(a,b) = 1$ and $\phi(b) = z$.
\end{corollary}
\begin{proof}
The subsurface $S \setminus a$ has two boundary components with winding number $0$ and so $\Arf_1(S \setminus a) = 1$.
Applying Lemma \ref{lem:existsGSB} we can pick some GSB on $S \setminus a$ with coprime winding numbers; let $T$ denote the subsurface filled by this pair of curves.
We can now pick any curve $b'$ disjoint from $T$ with $i(a,b') = 1$ .
Since $b'$ does not meet $T$ and $\Arf_1(\phi|_{T})=1$, we can apply Lemma \ref{lem:adjustwinding} to adjust $\phi(b')$ at will.    
\end{proof}

With these results in hand, we can now prove the desired transitivity statements.

\begin{proof}[Proof of Proposition \ref{prop:transitiveadm}]
Obviously transitivity on single curves follows from the result for pairs, but since the proof for pairs requires a bit of casework we will prove the result for single curves first as a demonstration of our techniques.

{\medskip \noindent \bf Single curves.}
Suppose first that $a, a' \subset S$ are both admissible. Complete $a$ to a GSB $a=a_1$, $b_1$, $\ldots, a_g, b_g$ of $S$. 
Using Corollary \ref{cor:specifytransverse}, there is some $b_1'$ on $S$ with $i(a',b_1')=1$ and $\phi(b_1') = \phi(b_1)$.
Now take the subsurface $Y'$ filled by $a'$ and $b_1'$ and consider its complement.
If $\phi|_{S \setminus Y'}$ is of spin type, then the additivity of the Arf invariant \cite[Lemma 2.11]{RW} implies that
\[\Arf(\phi|_{S \setminus Y'}) 
= \Arf(\phi) - \left( \phi(a')+1 \right) \left( \phi(b_1')+1 \right) 
= \sum_{i=2}^g \left( \phi(a_i)+1 \right) \left( \phi(b_i)+1 \right)  \mod 2.\]
Otherwise, it is not of spin type; in either case we can now apply Lemma \ref{lem:existsGSB} to find a GSB $a_2', b_2', \ldots, a_g', b_g'$ on $S \setminus Y'$ with
\[\phi(a_i)= \phi(a_i') \text{ and }\phi(b_i)= \phi(b_i') \text{ for all }
i.\]
By the usual change-of-coordinates principle (compare Lemma 2.3 of \cite{salter_higher_spin}), there is some $f \in \Mod(S)$ taking $a$ to $a'$, each $a_i$ to $a_i'$, and each $b_i$ to $b_i'$.
Since $f$ preserves the winding numbers of the curves of a GSB, it preserves the winding numbers of all simple curves (Lemma \ref{lemma:HJ}), and thus we see that $f \in \FMod(S, \phi)$.

{\medskip \noindent \bf Nonseparating pairs.}
If $g \ge 4$ and the admissible curves $a_1, a_2$ together do not separate $S$, then we can just repeat our argument for transitivity on single admissible curves: extend $a_1, a_2$ to an arbitrary GSB, use Corollary \ref{cor:specifytransverse} and \ref{lem:existsGSB} to extend $a_1', a_2'$ to a GSB with the same winding numbers, and then use the transitivity of the mapping class group action on GSBs to find some $f$ (necessarily in $\FMod(S, \phi)$) taking one GSB to the other.

If $g=3$ then we must be slightly more clever about how we choose our intial GSB since our choice of transverse curves $b_1$ and $b_2$ may constrain the winding numbers of the remaining curves $a_3$ and $b_3$ due to the $\Arf_1$ invariant.
Suppose first that $\phi$ is of spin type. Using Corollary \ref{cor:specifytransverse} twice, we can choose disjoint curves $b_1$ and $b_2$, each meeting their respective $a_i$ and disjoint from the other, such that 
\[
\Arf(\phi) + \phi(b_1) + \phi(b_2)=
0 \mod 2.\]
In particular, this implies that if we let $Y$ denote the (disconnected) subsurface obtained by taking a regular neighborhood of $a_1 \cup a_2 \cup b_1 \cup b_2$, then the contribution to $\Arf(\phi)$ of $\phi|_{S \setminus Y}$ must be 0, hence for any GSB $(a_3, b_3)$ on $S \setminus Y$ at least one of $\phi(a_3)$ or $\phi(b_3)$ must be odd.
Now we observe that
\[\text{sig}(\phi|_{S \setminus Y})=(\text{sig}(\phi), +1, +1)\]
and so $\Arf_1(\phi|_{S \setminus Y})$ is the $\gcd$ of an odd number and $2$, i.e., is $1$.

If $\phi$ is not of spin type then choose any disjoint $b_1$ and $b_2$, each meeting their respective $a_i$ and disjoint from the other, and define $Y$ similarly.
Then since some $\phi(\Delta_i)$ is even, the signature of $\phi|_{S \setminus Y}$ contains both an even number and $+1$, and so we see that $\Arf_1(\phi|_{S \setminus Y}) = 1$.
Therefore, no matter whether $\phi$ is of spin type or not, we can choose our $b_1$ and $b_2$ such that $\phi|_{S \setminus Y}$ has fixed $\Arf_1$, and so by Lemma \ref{lem:existsGSB} admits a GSB $a_3$, $b_3$ with $\phi(a_3)=0$ and $\phi({b_3})=1$.
We can now finish the proof by inserting a prime in all of the arguments above to get another GSB on $S$ with the same winding number data and then concluding as in the $g \ge 4$ case.

{\medskip \noindent \bf Separating pairs.}
Finally, suppose that $a_1 \cup a_2$ separates $S$ into two subsurfaces $T$ and $U$. In this case, neither of the complementary components to $a_1 \cup a_2$ is of spin type, so if $\phi$ is of spin type then we will need be somewhat clever about our choice of GSB to deal with the emergence of the Arf invariant.

Pick an arbitrary curve $b_1$ meeting $a_1$ and $a_2$ each exactly once. Since at least one of $T$ or $U$ has genus at least $2$ or genus 1 with $\Arf_1=1$, we can use Lemma \ref{lem:adjustwinding} to turn this curve into an admissible $b_1$ that also meets each of $a_1$ and $a_2$ exactly once.
Choose GSBs 
\[\mathcal B_{T} := s_1, t_1, \ldots, s_{g(T)}, t_{g(T)} 
\text{ for }T
\text{ and }
\mathcal B_{U} := u_1, v_1, \ldots, u_{g(U)}, v_{g(U)}
\text{ for }U\]
that are disjoint from $b_1$; then $\{a_1, b_1\} \cup \mathcal B_{T} \cup \mathcal B_{U}$ is a GSB for $S$.

Since $(a_1, a_2)$ and $(a_1', a_2')$ are in the same mapping class group orbit, there is a correspondence between their complementary components; let $T'$ and $U'$ denote the two components of $a_1' \cup a_2'$ corresponding to $T$ and $U$.
Since neither component is of spin type (having a boundary component with even winding number) or, if they have genus 1, have $\Arf_1 = 1$ with an admissible boundary component,
Lemma \ref{lem:existsGSB} implies that both $T'$ and $U'$ admit GSBs with any given tuples of winding numbers.
We may therefore choose GSBs $\mathcal B_{T'}$ and $\mathcal B_{U'}$ with the same winding numbers as those for $\mathcal B_{T}$ and $\mathcal B_{U}$. To extend these to a GSB of $S$, we just need to find an admissible curve disjoint from $\mathcal B_{T'} \cup \mathcal B_{U'}$ that meets $a_1'$ and $a_2'$ exactly once.

Suppose $\phi$ is of spin type. Then we see that for any choice of $b_1'$ meeting $a_1'$ exactly once and disjoint from $\mathcal B_T \cup \mathcal B_U$, we have
\begin{align*}
&\left(\phi(a_1) + 1\right)\left(\phi(b_1) + 1\right)
+ \sum_{g(T)}\left(\phi(s_i) + 1\right)\left(\phi(t_i) + 1\right)
+ \sum_{g(U)}\left(\phi(u_i) + 1\right)\left(\phi(v_i) + 1\right)
= \Arf(\phi)\\
= &\left(\phi(a_1') + 1\right)\left(\phi(b_1') + 1\right)
+ \sum_{g(T')}\left(\phi(s_i') + 1\right)\left(\phi(t_i') + 1\right)
+ \sum_{g(U')}\left(\phi(u_i') + 1\right)\left(\phi(v_i') + 1\right) \text{ mod }2
\end{align*}
which simplifies to $\phi(b_1) = \phi(b_1') \mod 2$ by our choices of $\mathcal B_{T'}$ and $\mathcal B_{U'}$.
Thus $\phi(b_1')$ must be even.
Now choose a curve $c$ on either $T'$ or $U'$ that 
\begin{itemize}
    \item is disjoint from $\mathcal B_{T'} \cup \mathcal B_{U'}$,
    \item meets $b_1'$ exactly once, and
    \item together with $a_1'$ bounds a surface of genus $1$ with 2 boundary components.
\end{itemize}
Such a $c$ can be obtained, for example, by taking the boundary of a regular neighborhood of $u_1' \cup v_1'$ and then connect summing that curve with $a_1'$. See Figure \ref{fig:admsep}.
By homological coherence (Lemma \ref{lemma:HJ}.\ref{item:HC}), it must be that $\phi(c)= \pm 2$ (where sign depends on orientation).
Twist-linearity (Lemma \ref{lemma:HJ}.\ref{item:TL}) then implies that some twist of $b_1'$ about $c$ will be admissible.
Thus the configurations of curves
\[a_1, b_1, a_2, \mathcal B_{T}, \mathcal B_{U}
\text{ and }
a_1', T_c^{-\phi(b_1')/2}(b_1'), a_2', \mathcal B_{T'}, \mathcal B_{U'}\]
have the same topological type, so there is a mapping class taking one to the other, and since all of the corresponding curves have the same winding number, any such mapping class must preserve $\phi$.

\begin{figure}[h]
    \centering
    \def\svgwidth{5in}
    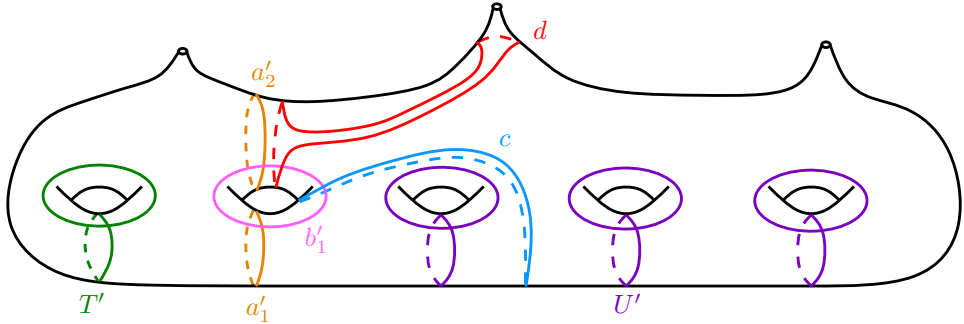
    \caption{GSBs and auxiliary curves as in the proof of Proposition \ref{prop:transitiveadm}.}
    \label{fig:admsep}
\end{figure}

If $\phi$ is not of spin type, then we can conclude by picking an arbitrary $b_1'$ disjoint from $\mathcal B_{T'} \cup \mathcal B_{U'}$. We then note that since $\phi$ is not of spin type, then there is some peripheral curve $\Delta_i$ with even winding number. Choose $c$ as before and let $d$ be a curve disjoint from all of the listed curves except $b_1'$, obtained by taking the connect sum of $a_2$ with this $\Delta_i$; by homological coherence again, its winding number must be odd. See Figure \ref{fig:admsep}.
Thus, by twisting around $c$ and $d$ we can change the winding number of $b_1'$ by any amount (while keeping all other winding numbers fixed) and so in particular $T_c^m T_d^n (b_1')$ is admissible for some $m, n$.
We can then conclude as in the spin case.
\end{proof}

\section{The admissible curve graph and its geometric model}\label{sec:models}

A \emph{graph of multicurves} for a surface $S$ is any graph whose vertices are multicurves on $S$. The simplest and most influential example is the \emph{curve graph} $\sC(S)$.
The curve graph  has all curves on $S$ as vertices and edges between two curves if and only if they intersect the fewest number of times possible for a pair of curves on $S$.
If $\xi(S) >1$ then edges correspond with disjointness, and when $\xi(S) =1$ the minimal intersection number is either 1 or 2.

We will focus on the following subset of the curve graph:
given a framing $\phi$ of $S$, the \emph{admissible curve graph} $\Cadm(S, \phi)$ relative to $\phi$ is the subgraph of $\sC(S)$ spanned by the non-separating curves that are admissible with respect to $\phi$. 

Proposition \ref{prop:transitiveadm} implies that the framed mapping class group $\FMod(S, \phi)$ acts transitively on the vertices of $\Cadm(S, \phi)$ and with finitely many orbits on its edges.
As a consequence of Lemma \ref{lem:genusadm},  every vertex of $\CS$ is distance 1 from a vertex of $\Cadm(S, \phi)$ when $\genus(S) \geq 2$. When $\genus(S) \geq 3$, Lemma \ref{lem:genusadm} also allows us to copy Salter's ``hitchhiking argument'' in the case of $r$-spin structures \cite[Lemma 3.11]{salter_higher_spin} to show $\Cadm(S, \phi)$ is connected.

\begin{lemma}\label{lem:Cadmconn}
	 If $\genus(S) \geq 3$, then for any framing on $S$, $\Cadm(S, \phi)$ is connected.
\end{lemma}
\begin{proof}[Proof sketch]
    The graph of genus 1 subsurfaces (with edges for disjointness) is connected \cite{Putman_connectivity}. Since each genus 1 subsurface contains an admissible curve, the paths in this graph can be upgraded to a path in $\Cadm(S, \phi)$.
\end{proof}

Given a graph of multicurves $\mathcal{X}$, a subsurface $W \subseteq S$ is a \emph{witness} for $\mathcal{X}$ if every vertex of $\mathcal{X}$ intersects $W$ and $\xi(W) < 0$. We let $\wit(\mathcal{X})$ denote the set of all witness for $\mathcal{X}$. 
For the admissible curve graph, the witnesses are all subsurfaces whose complement has no genus and where the winding numbers of the boundary curves do not satisfy a particular set of linear equations.

\begin{lemma}\label{lem:admwitnesses}
    Let $S = S_{g,n}$ with $g \geq 3$ and $n \geq 1$. Fix a framing $\phi$ of $S$.
    \begin{enumerate}
        \item If $Z \subseteq S$ is a genus 0 subsurface and $z_1,\dots,z_k$ are the boundary components of $Z$ and peripheral loops about its punctures, oriented such that  $Z$ is to the left of each $z_i$, then $Z$ contains a nonperipheral curve of winding number 0 if and only if there is no $I \subsetneq \{z_1,\dots,z_k\}$ such that 
\begin{equation}\label{eqn:witness wns}
\sum_{z \in I} \phi(z) = 1-|I|.
 \end{equation}
        \item A subsurface $W$ of $S$ is a witness for $\Cadm(S, \phi)$ if and only if each curve in $\partial W$ is not admissible and each component $Z$ of $S \setminus W$ is a genus 0 subsurface with the following property: enumerate the boundary components and peripheral loops of $Z$ as in the previous item. Then there is no $I$ such that \eqref{eqn:witness wns} holds and both $I$ and $\{z_i\}_{i=1}^k \setminus I$ contain curves of $\partial W$.
        \item If $V,W \in \wit(\Cadm(S, \phi))$ are disjoint, then each is a genus 0 subsurface that does not contain any admissible curves, and there does not exist $Z \in \wit(\Cadm(S, \phi))$ that is disjoint from both $V$ and $W$.
    \end{enumerate}
\end{lemma}

\begin{proof}
The first item is an immediate consequence of homological coherence and the fact that every curve on a genus $0$ surface separates it.
The second item follows from the first plus Lemma \ref{lem:genusadm}'s guarantee that every subsurface with genus contains an admissible curve; note that the condition that $\partial W$ meets both $I$ and $\{z_i\}_{i=1}^k \setminus I$ indicates whether or not a curve cutting off the boundaries $\{z_i\}_{i\in I}$ separates $S$ or not.
The third item is an immediate consequence of the second.
\end{proof}

Paralleling \cite{Vokes_HHS}, we now use the witnesses of a graph of multicurves to construct a ``model graph,'' which is in some sense the largest graph of multicurves that has the same witness set as the starting graph.

\begin{definition}
	Let $\fS$ be a collection of subsurfaces of $S$. We say $\fS$ is a set of \emph{valid witnesses} if for all $W \in \fS$,
	\begin{enumerate}
		\item $W$ is connected;
		\item $\xi(W) \geq 1$;
		\item  if $Z$ is a connected subsurface with $W \subseteq Z$, then $Z \in \fS$;
	\end{enumerate}
\end{definition}

\begin{definition} \label{definition: ksep}
	Let $\fS$ be a set of valid witnesses for the surface $S$.
	If $\fS = \emptyset$, define $\cK_{\fS}(S)$ to be a single point. Otherwise, define $\cK_{\fS}(S)$ to be the graph such that:
	\begin{itemize}
		\item each vertex is a multicurve $\gamma$ on $S$ with the property that each component of $S \setminus \gamma$ is \emph{not} an element of~$\fS$;
		\item  two multicurves $\gamma$ and $\delta$ are joined by an edge if either
		\begin{enumerate}
			\item $\gamma$ differs from $\delta$ by either adding or removing a single curve, or
			\item $\gamma$ differs from $\delta$ by ``flipping" a curve in some subsurface of $S$, that is, $\delta$ is obtained from $\gamma$ by replacing a curve $c \subset \gamma$ by a curve $d$, where $c$ and $d$ are contained in the same component $Y_c$ of $S \setminus (\gamma \setminus c)$ and are adjacent in $\sC(Y_c$). 
		\end{enumerate}
	\end{itemize}
\end{definition}

By construction, the set of witness for $\cK_\fS(S)$ is precisely $\fS$. Moreover, the vertex set of $\cK_{\fS}(S)$ is the maximal collection of multicurves whose set of witnesses is $\fS$. Thus, if $\cX$ is a graph of multicurves with $\wit(\cX) = \fS$, then the vertices of $\cX$ are a subset of $\cK_\fS(S)$. In the case of the admissible curve graph, this inclusion is Lipschitz.

\begin{lemma}\label{lem:inclusion_lipschtiz}
	If $\fS = \wit(\Cadm(S, \phi))$, then the inclusion $\Cadm(S, \phi) \to \cK_\fS(S)$ is $2$-Lipschitz
\end{lemma}

\begin{proof}
	If $a,b$ are a pair of disjoint admissible curves, then $a \cup b$ is also a vertex of  $\cK_\fS(S)$, hence $a, a \cup b, b$ is a path of length 2 connecting $a$ and $b$ in $\cK_\fS(S)$.
\end{proof}

Vokes studied the family of $\cK_\fS(S)$ as quasi-isometric models for graphs of multicurves. Specifically, she showed that if $\cX$ is a graph of multicurves on $S$ with a cobounded action of $\Mod(S)$ and no annular witnesses, then the inclusion $\cX \hookrightarrow \cK_\fS(S)$ for $\fS = \wit(\cX)$ is a quasi-isometry. 
The advantage of using $\cK_\fS(S)$ as a quasi-isometric model is that she showed that $\cK_\fS(S)$ is a \emph{hierarchically hyperbolic space} in a natural way. This means the coarse geometry of $\cK_\fS(S)$ can be well understood  using the subsurface projection machinery of Masur and Minsky and the relations between the subsurfaces in $\fS$; see \cite{BHS_HHSI,BHS_HHSII,Vokes_HHS} for full details. 

We note that while Vokes states her results in the case of an action of the full mapping class group, the only actual use of the action is in establishing the quasi-isometry described above. In particular, the proof in Section 3 of \cite{Vokes_HHS} as written demonstrates that $\cK_\fS(S)$ is a hierarchically hyperbolic space, even in the case where $\fS$ is not invariant under the mapping class group.

One consequence of Vokes's hierarchically hyperbolic structure is that Gromov hyperbolicity of the the graph is encoded in the disjointness of the witnesses.

\begin{theorem}[{Corollary 1.5 of \cite{Vokes_HHS}}]\label{cor:hyp_iff_no_disjoint}
	The graph $\cK_\fS(S)$ is Gromov hyperbolic if and only if $\fS$ does not contain a pair of disjoint subsurfaces.
\end{theorem}

\section{A quasi-isometry with the model}\label{sec:Cadm_qi_model}

Vokes's proof of the quasi-isometry between graphs of multicurves and their models relies on the action of the mapping class group in a fundamental way.
Specifically, given any connected graph of multicurves $\mathcal{X}$ that has no annular witnesses and has a cobounded action by $\Mod(S)$, she uses the ``change-of-coordinates'' principle and curve surgery arguments to build a quasi-isometry from $\cK_\fS(S)$  to $\mathcal{X}$, where $\fS$ is the set of witnesses of $\mathcal{X}$. 

In our setting, we only have access to the (weaker) framed versions of these techniques.
Moreover, there are infinitely many $\FMod(S, \phi)$ orbits of curves and of witnesses, so we cannot employ standard change-of-coordinates arguments of the form ``make a choice for each orbit, then propagate that choice around using the group action to get finiteness'' (e.g., \cite[Claim 4.3]{Vokes_HHS} or Lemma \ref{lem:int_bounds_dist_genus_sep} below).

Instead of relying on change-of-coordinates, we build our quasi-isometry $\cK_\fS(S) \to \Cadm(S, \phi)$ by going through an intermediary graph $\cG$, which admits a coarsely Lipschitz map $\Pi$ onto $\Cadm(S, \phi)$ (Lemma \ref{lem:septoadm_Lip}). 
One can then define a map $\Psi$ from $\cK_\fS(S)$ to subsets of $\cG$; while this map is not coarsely Lipschitz or even coarsely well-defined, the composition $\Pi \circ \Psi$ turns out to be (Proposition \ref{prop:comp_coarseLip}).

The utility of this approach is that $\cG$ admits an action of the entire mapping class group, so we can use standard change-of-coordinates arguments.
A fruitful comparison is the ``hitching a ride'' argument we used to show the connectivity of $\Cadm(S, \phi)$ in Lemma \ref{lem:Cadmconn}.

For the remainder of the section,  $S=S_{g,n}$ will be a surface with $g \geq 3$ and $n \geq 1$ and $\fS$ will be the set of witnesses for $\Cadm(S, \phi)$ with respect to a fixed framing $\phi$. Since we will only be considering theses graphs for the surface $S$, we will use $\Cadm$ and $\cK$ to denote $\Cadm(S, \phi)$ and $\cK_\fS(S)$ respectively.

\subsection{Coarse maps and quasi-isometries} Let $X,Y$ be metric spaces. A map $f \colon X \to 2^Y$ is {\em coarsely well-defined} if $f(x)$ has uniformly bounded diameter for every $x \in X$. It is {\em coarsely Lipschitz} if there are constants $K \geq 1$ and $C \geq 0$ such that
\[\text{diam}_Y( f(x) \cup f(x') ) \le K d_X(x, x') + C\]
for every $x, x' \in X$. In particular, note that coarsely Lipschitz maps are in particular coarsely well-defined.
Prototypical examples are  the inclusion of a connected subgraph into a connected graph, the subsurface projection map from the  the marking graph to $\sC(W)$ where $W \subseteq S$ is a subsurface, or the systole map that sends a point in Teichm{\"u}ller space to its hyperbolic systole(s). 

When $X$ is a graph, one can simply define a map $f\colon X \to 2^Y$ on the vertices and assume that the image of any point on an edge is the union of the images of the endpoints of that edge. In this case, to show $f$ is coarsely Lipschitz, it suffices to show that 
\begin{enumerate}
\item $f(x)$ is uniformly bounded for all vertices $x$ of $X$, and
\item if $x$ and $x'$ are two vertices joined by an edge of $X$, then $\diam(f(x) \cup f(x'))$ is uniformly bounded.
\end{enumerate}

Two spaces are \emph{quasi-isometric} if there exist two coarsely Lipschitz map $f \colon X \to 2^Y$ and $\overline{f} \colon Y \to 2^Y$ such that $d_X(x, \overline{f} \circ f(x))$ is uniformly bounded for all $x \in X$. In this case, $f$  is a \emph{quasi-isometry} from $X$ to $Y$  and $\overline{f}$ is  the \emph{quasi-inverse} of $f$. 

\subsection{The genus-separating curve graph}

We begin building our quasi-isometry from $\cK$ to $\Cadm$ by defining the intermediate graph $\cG$ that we use throughout this section.
We say that a separating curve $c \subseteq S$ is {\em genus-separating} if each component of $S \setminus c$ has positive genus.

\begin{definition}
	The {\em genus-separating curve graph} $\cG = \cG(S)$ is the graph whose vertices are genus-separating curves, and where two vertices are connected by an edge if the corresponding curves are disjoint.
\end{definition}

Putman's argument that the separating curve graph is connected in the closed case also shows that $\cG$ is connected \cite{Putman_connectivity}. The key commonality are that every vertex of $\cG$ is adjacent to a genus separating curve that cuts off a torus with one boundary component.
\begin{lemma}
	The graph $\cG$ is connected so long as $\genus(S) \ge 3$.
\end{lemma}

Since every subsurface with genus contains an admissible curve, we see that for any $c \in \cG$ both components of $S \setminus c$ are not witnesses for $\Cadm$. Thus $\cG$ is a subgraph of $\cK$.

\begin{remark}\label{rem:elecGtoK}
	While we will not use this in the sequel, we can in fact relate the geometries of $\cG$ and $\cK$ by considering their sets of witnesses.
	The witnesses for $\cG$ are exactly those subsurfaces that have genus $0$ complements, which form a strict superset of the witnesses for $\cK$ (characterized in Lemma \ref{lem:admwitnesses}).
	Using the ``factored space'' construction from \cite{BHS_HHS_AsDim}, we can thus view $\cK$ as being obtained from $\cK_{\wit(\cG)}(S)$ by coning off regions corresponding to the non-shared witnesses.
\end{remark}

As for the usual curve graph, intersection number  bounds distance in $\cG$.

\begin{lemma}\label{lem:int_bounds_dist_genus_sep}
	For each $n \geq 0$ there exists $N = N(n) \geq 0$ such that for any two genus-separating curves $c,d \in \cG$, if $i(c,d) \leq n$, then $d_{\cG}(c,d) \leq N$.
\end{lemma}

\begin{proof}
	By the change-of-coordinates principle in $\Mod(S)$, there exist finitely many pairs $\{(c_i,d_i)\}_{i=1}^k$ of genus-separating curves such that every pair of genus-separating curves that intersect at most $n$ times is in the $\Mod(S)$--orbit of some $(c_i,d_i)$.
	Setting $N =\max \{ d_\cG(c_i,d_i) :1 \leq i \leq k\}$,  the fact that $\Mod(S)$ acts by isometries on $\cG$ implies any two genus-separating curves that intersect at most $n$ times are at most $N$ far apart in $\cG$.
\end{proof}

\subsection{From genus-separating to admissible curves}

Define a map 
\[\Pi\colon \cG \to 2^{\Cadm}\]
by sending a genus-separating curve to the collection of admissible curves disjoint from it. This set is always non-empty by Lemma \ref{lem:genusadm}.

\begin{lemma}\label{lem:septoadm_Lip}
	The map $\Pi$ is coarsely Lipschitz.
\end{lemma}
\begin{proof}
	As remarked above, it suffices to check that the diameters of the images of vertices and edges are both bounded.
	
	Let $c \in \cG$ be any genus-separating curve and let $U,V$ denote the components of $S \setminus c$. 
	Let $a$ be any admissible curve in $\Pi(c)$, and assume without loss of generality that $a \subset U$. 
	Every admissible curve in $V$ is distance $1$ from $a$, and likewise every admissible curve in $U$ is disjoint from any curve in $V$.
	Thus $\Pi(c)$ has diameter $2$ as a subgraph of $\Cadm$.
	
	Now suppose $c$ and $d$ in $\mathcal G$ are disjoint; this implies that one of the (positive genus) components of $S \setminus c$ is nested inside a component of $S \setminus d$.
	In particular, this implies that $\Pi(c)$ and $\Pi(d)$ overlap, and since each has bounded diameter their union does as well.
\end{proof}

The map $\Pi$ is defined such that if $a \in \Cadm$ and $c \in \cG$ with $i(a,c) = 0$, then 
\[d_{\Cadm}(a, \Pi(c)) = 0.\]
Below, we prove a generalization of this fact that allows us to bound the distance between $a$ and $\Pi(c)$ by bounding the geometric intersection number $i(a,c)$.

\begin{lemma}\label{lem:adm_bdd_dist}
	For any $m\geq 0$, there exists $M = M(m) \geq 0$ such that for any admissible curve $a$ and any genus-separating curve $c$ with $i(a,c) \le m$, we have  
	$d_{\Cadm}(a, \Pi(c)) \le M$.
\end{lemma}

We will  only ever apply this lemma with $m=2$, but since the proof for general $m$ is not much harder we choose to include it here.

\begin{proof}[Proof of Lemma \ref{lem:adm_bdd_dist}]
	If $a$ is disjoint from $c$, then $a \in \Pi(c)$ and we are done.
	Otherwise, we will surger $c$ along $a$ to produce a new genus-separating curve $c'$ disjoint from $c$ that intersects $a$ strictly fewer times.
	By Lemma \ref{lem:int_bounds_dist_genus_sep}, this will allow us to decrease the intersection number of $a$ and $c$ at the cost of moving $c$ a fixed distance in $\mathcal G$.
	Since $\Pi$ is a coarsely Lipschitz map, this procedure moves the projection a uniformly bounded amount in $\Cadm$, proving the desired statement.
	
	Since $S$ has genus at least 3, there is at least one component $U_c \subset S \setminus c$ of genus at least 2.
	Consider an arc $\alpha$ of $a \cap U_c$. The regular neighborhood of $c \cup \alpha$ forms a pair of pants $P_{\alpha}$, one of whose boundaries is $c$; label the other two by $d$ and $e$. Because any strand of $a \cap U_c$ that meets $d$ or $e$ must travel through $P_\alpha$ while avoiding $\alpha$, any such strand must exit $P_\alpha$ through $c$. Thus, we have 
	\[i(a, d) + i(a,e) \le i(a,c)-2.\]
	
	If either $d$ or $e$ is separating,  then the other one is either separating or homotopic to a boundary curve of $S$ (they cannot both be homotopic to a boundary curve as $c$ is genus-separating).
	Since $U_c$ has positive genus, at least one of $d$ and $e$ is genus-separating; we then take $c'$ to be whichever is, completing the proof in this case.
	
	In the other case, $d$ and $e$ are both non-separating.
	Let $V_c \subset U_c$ denote the connected subsurface of $U_c \setminus (d \cup e)$ not containing $\alpha$.
	Choose an arc $\beta$ in $V_c$ connecting $d$ and $e$ that is disjoint from $a \cap V_c$.
	Such an arc always exists because either $a \cap V_c$ contains such an arc, or it does not, in which case one can take an arbitrary arc from $d$ to $e$ and surger it along its intersections with $a \cap V_c$ to make it disjoint; see Figure \ref{fig:adm_bdd_dist}.
	
	The curve $c'$ obtained from a regular neighborhood of $ d \cup e \cup \beta$ forms a pair of pants $P_\beta$ with $d$ and $e$. Since any arc of $a$ that enters $P_\beta$ through $c'$ cannot intersect $\beta$, that arc must exit through either $d$ or $e$. Thus $$i(c',a) \leq i(a,d) + i(a,e) <i(a,c).$$
	Since $c'$ is constructed to cut off a genus $\genus(U_c) - 1 \ge 1$ subsurface, we see that $c'$ is still genus-separating and is clearly disjoint from $c$.
	This completes the proof.
\end{proof}

\begin{figure}[ht]
	\centering
 \def\svgwidth{5.5in}
 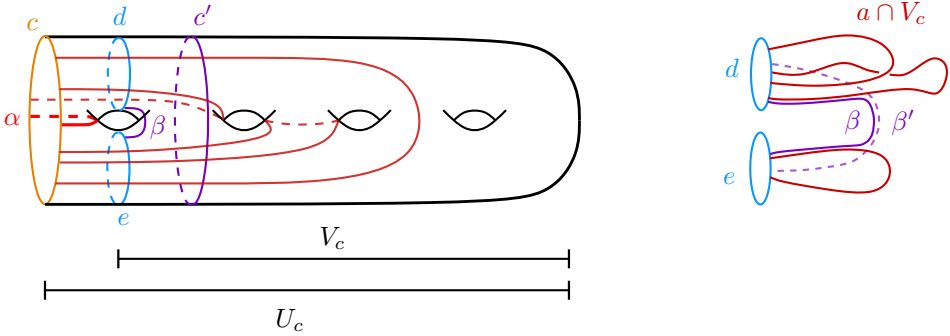
	\caption{On the left, the subsurfaces involved in the proof of Lemma \ref{lem:adm_bdd_dist}. On the right, surgering an arbitrary arc $\beta'$ from $d$ to $e$ along $a \cap V_c$ to obtain a disjoint arc $\beta$.}
	\label{fig:adm_bdd_dist}
\end{figure}

\subsection{A quasi-inverse}
We now construct a map $\Psi$ that sends vertices of $\cK$ to sets of genus-separating curves so that the composition $\Pi \circ \Psi$ is a quasi-inverse of the inclusion $\Cadm \to \cK$.  The idea to is assign a multicurve $\alpha \in \cK$ to the set of genus-separating curves that intersect the components of $S \setminus \alpha$ in a particularly nice way. This is always possible by the following lemma.

\begin{lemma}\label{lem:sep_curve_int_multicurve}
	For any multicurve $\alpha$ on $S$, there exists a genus-separating curve $c$         
	so that for each component $Y$ of $S \setminus \alpha$, we have exactly one of the following:
	\begin{enumerate}
		\item $c$ is disjoint from $Y$,
		\item $c \subseteq Y$,
		\item \label{item:end_pt_int} $c\cap Y $ is a single arc with both endpoints on the same curve of $\partial Y$, or
		\item \label{item:interior_pt_int} $c \cap Y$ is a pair of parallel arcs that both go from one curve $y_1 \in \partial Y$  to a different curve $y_2 \in \partial Y$.
	\end{enumerate}
\end{lemma}

\begin{proof}
	If a component of $S \setminus \alpha$ has positive genus, then the lemma is true using a separating curve cutting off that genus.
	Otherwise, the dual graph $D$ of $\alpha$ on $S$ must contain a cycle. We can use the dual graph to build such a  separating curve $c$ as follows:
	\begin{enumerate}
		\item Take any cycle $v_1, \ldots, v_n$ in the dual graph $D$ that meets any vertex of $D$ at most once. Let $a_i$ be the curve of $\alpha$/edge in the dual graph connecting $v_i$ to $v_{i+1}$ (where indices are taken mod $n$).
		
		\item On each subsurface $Y_i$ of $S \setminus \alpha$ corresponding to a vertex $v_i$ of the cycle, choose an arc $\beta_i$ connecting $a_{i-1}$ to $a_i$.
		\item The concatenation of the $\beta_i$ is now a curve $b$ that meets each $a_i$ exactly once.
		\item Set $c$ to be a regular neighborhood of $b \cup a_n$.
	\end{enumerate}
	By construction $c \cap Y_i$ is a pair of arcs parallel to $\beta_i$ for each $i \neq 1, n$, and it follows by inspection that $c \cap Y_1$ (and $c \cap Y_n$) is a single arc with both endpoints on $a_1$ (and $a_{n-1}$, respectively).
	See Figure \ref{fig:nicesep}.
\end{proof}

\begin{figure}[hb]
	\centering
\def\svgwidth{4.5in}
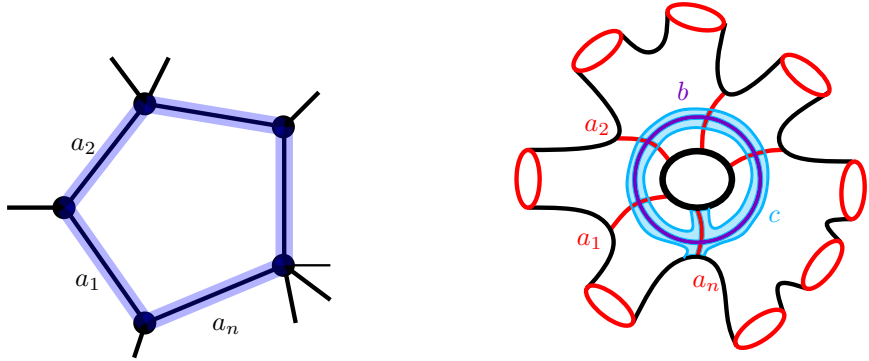
	\caption{Building a genus-separating curve out of a cycle in the dual graph.}
	\label{fig:nicesep}
\end{figure}

In light of Lemma \ref{lem:sep_curve_int_multicurve}, we define a map 
\[ \Psi \colon \cK \to 2^{\cG}\]
by setting $\Psi(\alpha)$ to be the set of genus-separating curves $c$ that satisfy the conclusion of Lemma \ref{lem:sep_curve_int_multicurve}.

Our discussion in Remark \ref{rem:elecGtoK} shows that this map is rather poorly behaved.
Viewing $\cK$ as (quasi-isometric to) the cone-off of (the model $\cK_{\wit(\cG)}(S)$ for) $\cG$, this map sends cone point points to entire product regions. In particular, the diameter of $\Psi(\alpha)$ need not be bounded.
Nevertheless, we will show that the composition $\Pi \circ \Psi$ is coarsely Lipschitz and is hence a quasi-inverse of the inclusion $\Cadm \to \cK$.

The key technical step is the next lemma, which takes a component $Y$ of $S \setminus \alpha$ and a genus-separating curve $c \in \Psi(\alpha)$ and produces an admissible curve $a$ that intersects $c$ at most 4 times and is disjoint from $Y$. This admissible curve provides an ``anchor'' that allows us to modify $c$ inside the component $Y$ without large changes in the eventual composition $\Pi \circ \Psi(\alpha)$. It is in this lemma where we need the finer control over the genus-separating curve in $\Psi(\alpha)$ ensured by Lemma \ref{lem:sep_curve_int_multicurve} as opposed to defining $\Psi(\alpha)$ to be all genus-separating curves that intersect  each curve of $\alpha$ some fixed number of times.

\begin{lemma}\label{lem:anchoring_curve}
	Let $\alpha$ be a multicurve in $\cK$ and $c \in \Psi(\alpha)$. For each component $Y$ of $S\setminus\alpha$  that $c$ intersects, there exists an admissible curve $a_Y$ that is disjoint from $Y$ and has $i(c,a_Y) \leq 4$.
\end{lemma}

\begin{proof}
	Let $Y$ be a component of $S \setminus \alpha$ that $c$ intersects.  If any curve of $\alpha$ is admissible, then $c$ intersects that curve at most twice  and we are done.
    This also allows us to proceed by assuming that $S \setminus \alpha$ is disconnected: because each component of $S \setminus \alpha$ is not a witness, if $S \setminus \alpha$ is connected then $\alpha$ must contain an admissible curve.
	
	Since $Y$ is not a witness for $\Cadm$ by the definition of $\cK$, some component $Z$ of $S \setminus Y$ contains an admissible curve. If $c$ is disjoint from $Z$, then $c$ is disjoint from the admissible curve on $Z$ and again we are done.
	So suppose that $c$ intersects $Z$; then $c \cap Z$ separates $Z$ since $c$ is separating.  Since $c$ is genus-separating, if $Z$ has positive genus then at least one of the components  of $Z - (c \cap Z)$ must also have genus. Applying Lemma \ref{lem:genusadm}, this implies there is an admissible curve in $Z$ that is disjoint from $c$ whenever $Z$ contains genus. 
	
	We can therefore concentrate on the case where $Z$ has no genus. In this case, every curve on $Z$ is separating, and which curves of $Z$ are admissible are determined by how they separate the boundary components and peripheral curves of $Z$ (Lemma \ref{lemma:HJ}.\ref{item:HC}).
	Let $A$ be a set of boundary and peripheral curves of $Z$ such that any curve partitioning the boundaries and peripheral curves into $A$ and its complement must be admissible.
	We argue below that one can always draw a curve $a$ that cuts off the curves of $A$ and intersects $c$ at most 4 times. 
	
	To facilitate this, we first show that $c \cap Z$ cuts $Z$ into at most 3 components.
	Since $c$ intersects at most 2 components of $\partial Y$, it also intersects at most $2$ components of $\partial Z$ (and intersects each component at most twice) and must be disjoint from all peripheral curves.
	If $c$ intersects exactly one component of $\partial Z$, then we are in case 3 of Lemma \ref{lem:sep_curve_int_multicurve} and so $c \cap Z$ must be a single arc with both endpoints on the same boundary component of $Z$; in this case $Z -(c\cap Z)$ has two components.
	When $c$ intersects two distinct components $z_1,z_2$ of $\partial Z$, then we are in case 4 of Lemma \ref{lem:sep_curve_int_multicurve} and so $c \cap Z$ is a pair of  arcs $c_1,c_2$ such that either 
	\begin{itemize}
		\item both endpoints of $c_i$ are on $z_i$ for each $i \in \{1,2\}$, or
		\item $c_1,c_2$ are parallel arcs each running from $z_1$ to $z_2$.
	\end{itemize}
	In the first case, $Z - ( c \cap Z)$ has either two or three components and in the second it has two.
	
	To find an admissible curve on $Z$ that intersects $c$ at most $4$ times, let $Z_1, Z_2,Z_3$ be the components of $Z - (c\cap Z)$, with $Z_3$ being omitted in the case of two components. Without loss of generality, assume $\partial Z_2$ contains an arc of $c \cap Z$ in common with both $\partial Z_1$ and $\partial Z_3$ when there are three components. 
	Partition the curves in $A$ into $5$ (possibly empty) sets: $A_1,A_2,A_3$ and $B_1,B_2$. 
	The $A_i$ are the subsets of curves in $A$ that are contained in $Z_i$ for each $i$, while $B_1$ are the curve(s) that contains the endpoints of the arc in $c \cap Z$ shared by $\partial Z_1$ and $\partial Z_2$ and $B_2$ is the same for $\partial Z_2$ and $\partial Z_3$ (when $Z_3$ exists). Note that the $A_i$ may contain curves peripheral to the punctures, but the $B_i$ must always consist of essential curves on $S$.
	
	Order the curves in each $A_i$ and $B_i$ in any sequence, then join successive curves by disjoint arcs in the following order, skipping any empty sets: $A_1$, $B_1$, $A_2$, $B_2$, $A_3$.
	We further stipulate that the arcs must be disjoint from $c \cap Z$ unless some set is empty, in which case their intersection with $c \cap Z$ is allowed to be the difference of the indices of the $Z_i$ that the two sets border.
	For example, if only $A_2$ is empty then the arc from $B_1$ to $B_2$ must still be disjoint from $c$, since both $B_1$ and $B_2$ border $Z_2$, but if $B_1$, $A_2$, and $B_2$ are empty then the arc from $A_1$ to $A_3$ is allowed to meet $c \cap Z$ twice.
	Compare Figure \ref{fig:cutoff}.
	
	\begin{figure}
		\centering
    \def\svgwidth{3.5in}
    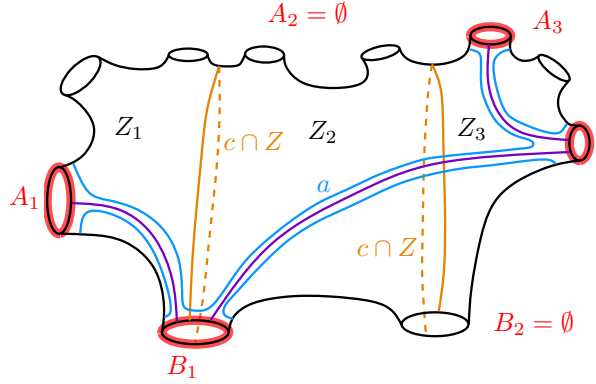
		\caption{Building a curve that cuts off $A$, and is hence admissible. The highlighted curves are in $A$. In this example, $A_2$ and $B_2$ are empty, so the arc from $B_1$ to $A_3$ meets $c \cap Z$ exactly once.}
		\label{fig:cutoff}
	\end{figure}
	
	A regular neighborhood of $A$ together with these arcs produces a curve $a$ that cuts off all of the curves in $A$, and hence must be admissible.
	It remains to note that the arcs and curves in the construction of $a$ are all disjoint from $c \cap Z$ except for the $B_i$'s and arcs that travel between different $Z_i$'s (which exist only when one of the $B_i$'s is empty). 
	In particular, this means that $a$ intersects $c$ only in a neighborhood of the $B_i$ or the above-mentioned arcs, and only does so at most twice for each component of the construction.
\end{proof}

We now prove that $\Pi \circ \Psi(\alpha)$ has uniformly bounded diameter for each $\alpha \in \cK$. The proof will use Lemma \ref{lem:anchoring_curve} to anchor the image of $\Pi \circ \Psi(\alpha)$ while we modify the genus-separating curves on the components of $S \setminus \alpha$ to reduce intersection numbers.

\begin{proposition}\label{prop:psi_to_pi}
	There is an $N \geq 0$ such that for any $\alpha \in \cK$ and $c,d \in \Psi(\alpha)$, there is  $c' \in \Psi(\alpha)$ with
	\begin{enumerate}
		\item $i(c',d) \leq 2 | \chi(S)|$ and 
		\item The diameter of $\Pi(c) \cup \Pi(c')$ in $\Cadm$ is at most $N$.
	\end{enumerate} 
	In particular, $\Pi \circ \Psi(\alpha)$ has uniformly bounded diameter for all $\alpha \in \cK$.
\end{proposition}
\begin{proof}
	Throughout the proof, we fix representatives of the isotopy classes of all of the curves involved such that $c$ and $d$ are each in minimal position with respect to $\alpha$, and such that no points of $c \cap d$ lie on $\alpha$.
	This allows us to give meaning to statements like ``$c$ and $d$ intersect on a component $Y$ of $S \setminus \alpha$'' even though there is no canonical minimal position for triples of isotopy classes of curves.
	
	Having fixed representatives, the proposition will follow by inductively applying the following claim.
	
	\begin{claim}\label{claim:adjust_inside_Y}
		If $Y$ is a component of $S \setminus \alpha$ on which $c$ and $d$ intersect, then there exists $c_Y \in \Psi(\alpha)$ such that $c_Y$ and $d$ intersect at most twice on $Y$ and $c_Y$ agrees with $c$ on $S \setminus Y$.
	\end{claim}

\begin{proof}
We will show that $c_Y$ can be obtained by replacing $c\cap Y$ with some well chosen arcs that intersect $d\cap Y$ at most twice.
		By construction, each of $c \cap Y$ and $d \cap Y$ is either a single arc connecting a boundary component to itself (which necessarily separates $Y$) or a pair of parallel arcs connecting different boundary components (and neither of these arcs can separate $Y$).  
		
		We first handle the case where $c \cap Y$ is a pair of parallel arcs. Let $c_1^1,c_1^2, c_2^1,c_2^2$ be the four 
		endpoints of $c\cap Y$ in $Y$ such that $c_i^1$ is joined by an arc of $c\cap Y$ to $c_i^2$.  If $d \cap Y$ is a single arc, then $c_i^1$ and $c_i^2$ are either on the same or different sides of $ d \cap Y$. In either case, we can connect each $c_i^1$ to its corresponding $c_i^2$ with an arc $\gamma_i$ such that $\gamma_1$ and $\gamma_2$ are parallel arcs and $i(\gamma_i,d) \leq 1$.
  If $d\cap Y$ is instead a pair of parallel arcs, let $\delta_1,\delta_2$ be the arcs of $d\cap Y$. Now $Y \setminus \delta_1$ is connected, but  $(Y \setminus \delta_1) \setminus \delta_2$ has two components. Thus $c_i^1$ and $c_i^2$ are either on the same or different sides of of $\delta_2$ in $Y\setminus \delta_1$.
  As before, this means we can connect each pair $c_i^1$ and $c_i^2$ with an arc $\gamma_i$ such that $\gamma_1$ and $\gamma_2$ are parallel, $i(\gamma_i,\delta_2) \leq 1$, and $i(\delta_1,\gamma_i) = 0$.
  In either case, let $c_Y$ be the curve obtained from $c$ be replacing $c\cap Y$ with $\gamma_1 \cup \gamma_2$. Since $c\cap Y$ and $c_Y \cap Y$ are both parallel arcs between the same boundary components of $Y$, we see that $S \setminus c$ is homeomorphic to $S \setminus c_Y$, and in particular $c_Y$ is genus-separating.
  By construction, it is also clear that $c_Y \in \Psi(\alpha)$, so we are done.
		
Now consider the case where $c \cap Y$ is a single arc. Since $c\cap Y$ separates $Y$, we orient $c$ and then label each boundary component and peripheral curve of $Y$ by ``left'' or ``right'' depending on which side of $c \cap Y$  it lies on. 
Let $g_{l}$ and $g_r$ be the genus of the left and right sides of $Y  \setminus (c\cap Y)$ respectively. 
We will find $c_Y$ by replacing $c\cap Y$ with an arc $\gamma$ that separates $Y$ into two components, one with genus $g_{l}$ and all the left curves of $Y$ and the other with genus $g_{r}$ and all the right curves of $Y$ (any such arc is essential on $Y$ since $c \cap Y$ is an essential arc and $\gamma$ will separate $Y$ in the same way as $c$).
This ensures $S \setminus c$ is homeomorphic to $S \setminus c_Y$, which makes $c_Y$ a genus-separating curve which is in $\Psi(\alpha)$ by construction. Let $c_1,c_2$ be the end points of $c \cap Y$ in $\partial Y$.

\begin{figure}[hb]
		\centering
            \def\svgwidth{2.5in}
\begingroup%
  \makeatletter%
  \providecommand\color[2][]{%
    \errmessage{(Inkscape) Color is used for the text in Inkscape, but the package 'color.sty' is not loaded}%
    \renewcommand\color[2][]{}%
  }%
  \providecommand\transparent[1]{%
    \errmessage{(Inkscape) Transparency is used (non-zero) for the text in Inkscape, but the package 'transparent.sty' is not loaded}%
    \renewcommand\transparent[1]{}%
  }%
  \providecommand\rotatebox[2]{#2}%
  \newcommand*\fsize{\dimexpr\f@size pt\relax}%
  \newcommand*\lineheight[1]{\fontsize{\fsize}{#1\fsize}\selectfont}%
  \ifx\svgwidth\undefined%
    \setlength{\unitlength}{187.02168964bp}%
    \ifx\svgscale\undefined%
      \relax%
    \else%
      \setlength{\unitlength}{\unitlength * \real{\svgscale}}%
    \fi%
  \else%
    \setlength{\unitlength}{\svgwidth}%
  \fi%
  \global\let\svgwidth\undefined%
  \global\let\svgscale\undefined%
  \makeatother%
  \begin{picture}(1,0.65421338)%
    \lineheight{1}%
    \setlength\tabcolsep{0pt}%
    \put(0,0){\includegraphics[width=\unitlength,page=1]{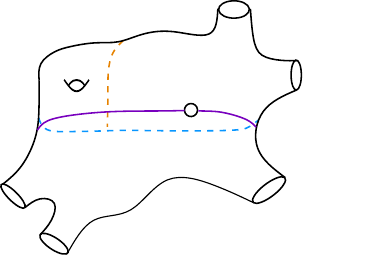}}%
    \put(0.68641907,0.30953572){\color[rgb]{0.47058824,0,0.74509804}\makebox(0,0)[lt]{\lineheight{1.25}\smash{\begin{tabular}[t]{l}$d\cap Y$\end{tabular}}}}%
    \put(0,0){\includegraphics[width=\unitlength,page=2]{left_right_part_1_svg-tex.pdf}}%
    \put(0.02607904,0.3578173){\color[rgb]{0.03137255,0.58823529,1}\makebox(0,0)[lt]{\lineheight{1.25}\smash{\begin{tabular}[t]{l}$p_1$\end{tabular}}}}%
    \put(0.02236132,0.29584867){\color[rgb]{0.03137255,0.58823529,1}\makebox(0,0)[lt]{\lineheight{1.25}\smash{\begin{tabular}[t]{l}$p_2$\end{tabular}}}}%
    \put(0.34862518,0.18841722){\color[rgb]{0.90196078,0.50980392,0}\makebox(0,0)[lt]{\lineheight{1.25}\smash{\begin{tabular}[t]{l}$\gamma_2$\end{tabular}}}}%
    \put(0.35758439,0.45080674){\color[rgb]{0.90196078,0.50980392,0}\makebox(0,0)[lt]{\lineheight{1.25}\smash{\begin{tabular}[t]{l}$\gamma_1$\end{tabular}}}}%
  \end{picture}%
\endgroup%

		 	\caption{The curves $p_1,p_2$ cobounding the pair of pants $P$. The arcs $\gamma_1$ and $\gamma_2$ cut $S \setminus P$ into ``left'' and ``right'' sides.}
		 	\label{fig:left_right_1}
\end{figure}
		
If $d \cap Y$ is a single arc, let $y$ be the curve of $\partial Y$ that $d$ intersects. The boundary of a neighborhood of $(d\cap Y) \cup y$ is a pair of curves $p_1,p_2$ that cobound a pair of pants $P$ with the boundary curve $y$. The complement $Y \setminus P$ has two components $Z_1,Z_2$ where $Z_i$ contains $p_i$ as a boundary curve; see Figure \ref{fig:left_right_1}.

Suppose that $c$ also intersects the boundary curve $y$. On each $Z_i$, we can draw an arc $\gamma_i$ with both endpoints on $p_i$ such that $\gamma_i$ separates $Z_i$ into two components, one that contains the left boundary components of $Y$ that also live on $Z_i$ and the other that contains the right boundary components. Moreover, we can choose the $\gamma_i$ such that the sum of the genera on the ``left'' sides of $Z_i \setminus \gamma_i$ is $g_l$ and the sum of the genera on the ``right'' sides is $g_r$.
The $\gamma_i$ also separate $p_i$ into ``left'' and ``right'' arcs.

We can now complete $\gamma_1\cup \gamma_2$ to an arc on all of $Y$ by adding arcs in the pair of pants $P$.
Select three disjoint arcs $a,b_1,b_2$ such that $a$ joins one endpoint of $\gamma_1$ to one endpoint of $\gamma_2$ and each $b_i$ joins the other endpoint of $\gamma_i$ to $c_i$ by an arc in $P$.
These arcs can be chosen such that $a$ intersects $d\cap Y$ once, $b_1$ is disjoint from $d \cap Y$, and $b_2$ intersects $d\cap Y$ at most once.
Moreover, we can choose these arcs such that the left arcs of $p_i$ are in one component of $P \setminus (a \cup b_1 \cup b_2)$ and the right arcs are in the other; see Figure \ref{fig:pants_arcs}. 
The desired arc $\gamma$ is the concatenation of $\gamma_1$, $\gamma_2$ and these arcs in $P$.

\begin{figure}[h]
		\centering
\def\svgwidth{5in}
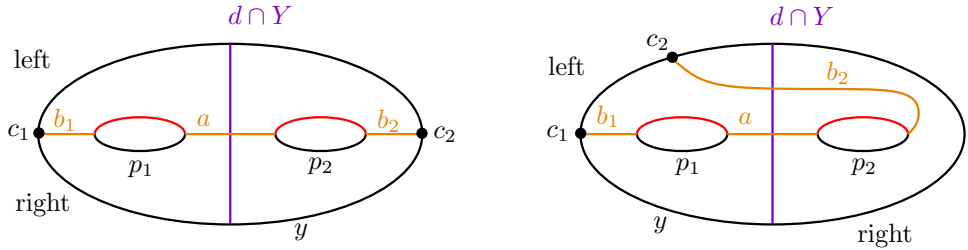
\caption{The arcs $a,b_1,b_2$ one must add in the pair of pants $P$ to complete $\gamma_1 \cup \gamma_2$ to $\gamma$.}
		 	\label{fig:pants_arcs}
\end{figure}

The case when $c$ does not intersect the boundary curve $y$ is similar. In this case $c$ intersects a different boundary curve $y' \in \partial Y$ and without loss of generality, $y' \subset Z_2$. We draw $\gamma_1$ as we did in the previous case, but instead of $\gamma_2$, we draw two arcs $\gamma_2^1, \gamma_2^2$ where $\gamma_2^1$ connects $c_1$ to $p_2$ and $\gamma_2^2$ connects $c_2$ to $p_2$ such that $\gamma_2^1\cup \gamma_2^2$ cuts $Z_2$ into two pieces with the appropriate boundary components and number of genus on the ``left'' and ``right'' sides.
We now finish $\gamma$ by joining each end point of $\gamma_2^i$ on $p_2$ to one of the endpoint of $\gamma_1$ on $p_1$ by arcs in $P$ that intersect $d\cap Y$ exactly once and separate the left and right arc of $p_1,p_2$ to the correct sides.
		
Now suppose $d \cap Y$ is a pair of parallel arcs between two boundary component $y_1,y_2 \in \partial Y$. There is a unique curve $p \subset Y$ that forms a pair of pants $P$ with $y_1$ and $y_2$ such that $P$ contains $d \cap Y$; this curve $p$ is found by taking the boundary of a neighborhood of $(d\cap Y) \cup y_1 \cup y_2$. Note that $Y \setminus P$ is a connected subsurface with the same genus as $Y$ but one fewer boundary. 

Assume first that both $y_1$ and $y_2$ are on the same side of $c \cap Y$; this implies $c$ is disjoint from $y_1$ and $y_2$.   Since $g(Y) = g(Y \setminus P)$ and $y_1,y_2$ are on the same side of $c\cap Y$, we can draw an arc $\gamma$ on $Y \setminus P$ with connects $c_1$ to $c_2$ and cuts $Y$ into two components, one with $g_l$ genus and all the ``left'' components of $\partial Y$ and one with $g_r$ genus and all the ``right'' components. 

Now assume that both $y_1$ and $y_2$ are on different sides of $c \cap Y$ (again this implies $c$ is disjoint from $y_1$ and $y_2$). Without loss of generality let $y_1$ be on the left side of $c$ and $y_2$ on the right. In this case we draw two arcs $\gamma_1,\gamma_2$ on $Y \setminus P$ such that
$\gamma_1$ connects $c_1$ to $p$, $\gamma_2$ connects $c_2$ to $p$, and $\gamma_1\cup \gamma_2$ separates $Y \setminus P$ into  ``left'' and ``right'' components where the left component has $g_l$ genus and all the  left curves of $Y$ except $y_1$ and the right component has $g_r$ genus and all the right curves except $y_2$.
We complete $\gamma_1\cup \gamma_2$ to the arc $\gamma$ on $Y$ by joining $\gamma_1$ to $\gamma_2$ by an arc in $P$ that separates $y_1$ and $y_2$ to the correct side of $Y \setminus \gamma$; this can be done such that the final arc has $i(\gamma, d \cap Y) \le 2$; see Figure \ref{fig:pants_arcs_2}.

\begin{figure}[ht]
		\centering
            \def\svgwidth{5in}
            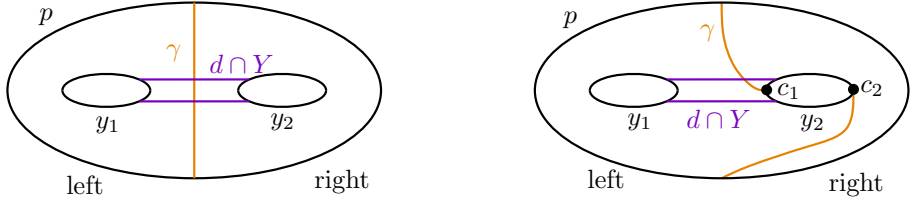
		
		 	\caption{ The arc drawn in $P$ to complete the arc $\gamma$. One the left, the case where $y_1$ and $y_2$ are on different sides of $c \cap Y$. On the right, the case where $c$ intersects $y_2$.}
		 	\label{fig:pants_arcs_2}
\end{figure}

Finally, assume that $c$ intersects exactly one of $y_1$ or $y_2$. Without loss of generality, assume $c$ intersects $y_2$ and $y_1$ is on the left side of $c$.
As in the previous cases, pick an arc $\gamma_0$ on $Y \setminus P$ that has both endpoints on $p$ and separates $Y \setminus P$ into two components where the ``left'' component has $g_l$ genus and contains all left curves of $Y$ except $y_1$ and the ``right'' component has $g_r$ genus and contains all right curves.
We complete $\gamma_0$ to an arc $\gamma$ on $Y$ by joining the endpoints of $\gamma_0$ to $c_1$ and $c_2$ by arcs in $P$ that separate $y_1$ to the ``left'' side of $Y \setminus \gamma$; this can be done  such that the final arc has $i(\gamma,d \cap Y) \leq 2$ ; see Figure \ref{fig:pants_arcs_2}.

We conclude by observing that in any of the three above cases, we have produced an arc $\gamma$ on $Y$ with the same topological type as $c \cap Y$ but that intersects $d$ at most twice on $Y$. Surgering $c$ along $\gamma$ as before we produce the desired curve $c_Y$.
\end{proof}

To prove  Proposition \ref{prop:psi_to_pi}, let $Y_1,\dots Y_k$ be the components of $S \setminus \alpha$ on which $c$ and $d$ intersect.
Applying Claim \ref{claim:adjust_inside_Y} to $Y_1$, we get a genus-separating curve $c_1 \in \Psi(\alpha)$ that intersects $d$ at most 4 times in $Y_1$ and agrees with $c$ outside of $Y_1$. By Lemma \ref{lem:anchoring_curve}, there is an admissible curve $a_1$ on $S \setminus Y_1$ that intersects $c$, and hence $c_1$, at most twice. 
Applying Lemma \ref{lem:adm_bdd_dist}, this implies that $a_1$ is $M$-close to both $\Pi(c)$ and $\Pi(c_1)$ in $\Cadm$ for some universal $M$. Hence,  $\Pi(c)$ and $\Pi(c_1)$ are $2M$-close to each other.
Repeating this argument, we produce a sequence of genus-separating curves $c = c_0,c_1,\dots, c_k$ in $\Psi(\alpha)$ such that $\Pi(c_i)$ and $\Pi(c_{i+1})$ are $2M$-close in $\Cadm$ and $i(c_k,d)$ is at most $2$ times the number of components of $S\setminus\alpha$, which is at most $|\chi(S)|$. 
The final curve $c_k$ is the desired curve $c'$.
	
	We now establish the requisite diameter bounds.
	Since the length of the sequence from $c$ to $c'$ is bounded by $|\chi(S)|$, each $\Pi(c_i)$ has uniformly bounded diameter in $\Cadm$, and each $\Pi(c_i)$ and $\Pi(c_{i+1})$ are  $2M$-close, we conclude that $\Pi(c) \cup \Pi(c')$ has uniformly bounded diameter. This gives (2).
	
	Finally, $c'$ and $d$ have uniformly bounded intersection number by construction, so by Lemma \ref{lem:int_bounds_dist_genus_sep} they have uniformly bounded distance in $\cG$.
	Since $\Pi$ is coarsely Lipschitz (Lemma \ref{lem:septoadm_Lip}), we see that $\Pi(c') \cup \Pi(d)$ also has uniformly bounded diameter. 
	The last statement of Proposition \ref{prop:psi_to_pi} now follows by the triangle inequality.
\end{proof}

We now show that the admissible curve graph $\Cadm$ is quasi-isometric to the model $\cK$. Since the inclusion $\Cadm \to \cK$ is simplicial and hence 1-Lipschitz, this statement is implied by the following:

\begin{proposition}\label{prop:comp_coarseLip}
	The map  $\Pi \circ \Psi \colon \cK \to \Cadm$ is a quasi-inverse to the inclusion $\Cadm \to \cK$.
\end{proposition}

\begin{proof}
	We first check that for all $a \in \Cadm$, the image $\Pi \circ \Psi(a)$ is uniformly close to $a$ in $\Cadm$. 
	Since $\genus(S) \geq 3$, there must exists a genus-separating curve $c$ disjoint from $a$. Hence $c \in \Psi(a)$ and $a \in \Pi(c)$. Thus $a \in \Pi\circ\Psi(a)$ as desired.
	
	We now show that $\Pi \circ \Psi$ is coarsely Lipschitz; this will complete the proof of  Proposition \ref{prop:comp_coarseLip}. We have already shown in Proposition \ref{prop:psi_to_pi} that the image of every vertex of $\cK$ has uniformly bounded diameter, so it suffices to do the same for every edge.
	That is, if $\alpha,\alpha' \in \cK$ are two vertices joined by an edge, then we must show that
	\[\diam(\Pi\circ\Psi(\alpha) \cup \Pi \circ \Psi(\alpha'))\]
	is uniformly bounded.
	
	If the edge from $\alpha$ to $\alpha'$ corresponds to adding a curve to $\alpha$ to achieve $\alpha'$, then $\Psi(\alpha') \subseteq \Psi(\alpha)$ by definition. This implies $\Pi\circ \Psi(\alpha') \subseteq \Pi \circ \Psi(\alpha)$; the desired diameter bound then follows from Proposition \ref{prop:psi_to_pi}.
	
	Now assume the edge from $\alpha$ to $\alpha'$ corresponds to a flip move. 
	Let $x \in \alpha$ and $x' \in \alpha'$ such that $x$ is flipped to $x'$.
	If $x$ and $x'$ are disjoint, then $\alpha \cup x'$ is a vertex of $\cK$ as adding curves to a vertex of $\cK$ always produces a new vertex of $\cK$. Now $\alpha \cup x'$ is joined by an  edge to both $\alpha$ and $\alpha'$ as removing $x'$ produces $\alpha$ and removing $x$ produces $\alpha'$. The desired bound
	now follows from the proceeding paragraph about add/remove edges.
	
	If $x$ and $x'$ are not disjoint, then the component $Y$ of $S \setminus (\alpha \setminus x)$ that contains $x$ has $\xi(Y) =1$. If $Y$ is not a witness, then $\alpha \setminus x = \alpha'  \setminus x'$  is a vertex of $\cK$ that is joined by an add/remove-edge to both $\alpha$ and $\alpha'$.  As before this establishes the bound. 
	
	If $Y$ is a witness, then  Lemma \ref{lem:admwitnesses} requires $S \setminus Y$ has no genus. Since $\xi(Y) = 1$ and $\genus(S) \ge 3$, this is only possible if $\genus(S) = 3$  and $Y$ is a 4-holed sphere where every curve in $\partial Y$ is non-peripheral and non-separating on $S$. In this case, $x$ and $x'$ intersect twice in the 4-holed sphere $Y$.
 Thus, flipping $\alpha$ to  $\alpha'$ corresponds to moving from the dual graph $D$ for $\alpha$ to the dual graph $D'$ for $\alpha'$ by performing a ``Whitehead move'' where one collapses the edge of $D$ dual to  $x$ and then expands an edge dual to $x'$; see Figure \ref{fig:4-holed_sphere}. Since no curves in $\partial Y$ are separating or peripheral on $S$, the dual graph $D$ contains a cycle $C$ with an edge dual to $x$ such that after performing the Whitehead move to produce $D'$, the cycle $C$ becomes a cycle $C'$ of $D'$ that does not include the edge dual to $x'$.  
There is therefore a genus-separating curve $c$ built from $C$ that will be disjoint from $x'$, which implies $c \in \Psi(\alpha) \cap \Psi(\alpha')$.
Since $\Pi(c)$  will then be contained in  $\Pi \circ \Psi(\alpha) \cap \Pi \circ \Psi(\alpha')$, we have that $\diam(\Pi \circ \Psi(\alpha) \cup \Pi \circ \Psi(\alpha'))$ is uniformly bounded by Proposition \ref{prop:psi_to_pi}. 
\end{proof}

\begin{figure}[ht]
    \centering
    \def\svgwidth{5in}
    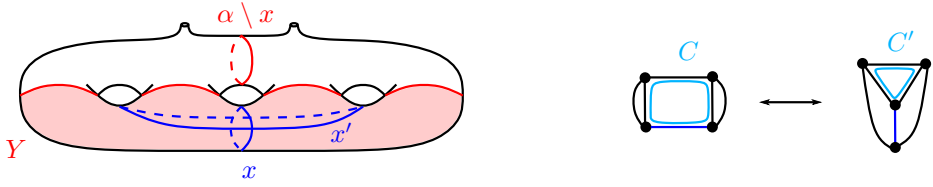
    \caption{One the left, the subsurface $Y$ where $x$ is flipped to $x'$. One the right, the Whitehead move on the dual graph corresponding to flipping $x$ to $x'$. The cycle $C$ is sent to the cycle $C'$ under this move.
    }
    \label{fig:4-holed_sphere}
\end{figure}

\begin{proof}[Proof of Theorem \ref{mainthm:CadmHHS}]
Lemma \ref{lem:inclusion_lipschtiz} and Proposition \ref{prop:comp_coarseLip} together show that $\Cadm$ is quasi-isometric to the hierarchically hyperbolic space $\cK$. Since hierarchical hyperbolicity  can be passed along quasi-isometries, $\Cadm$ is also hierarchically hyperbolic.

As Gromov hyperbolicity is also a quasi-isometry invariant, it suffices to  to verify that $\cK$ is not Gromov hyperbolic. By Corollary \ref{cor:hyp_iff_no_disjoint}, $\cK$ is not Gromov hyperbolic if and only if $\Cadm$ has a pair of disjoint witnesses.
Let $\Delta_1,\dots,\Delta_n$ be peripheral curves encircling the punctures of $S$. Without loss of generality, assume $\phi(\Delta_i) \geq 0$ for $i \in \{1,\dots, k\}$ and $\phi(\Delta_i) <0$ for $i \in \{k+1,\dots,n\}$. 
Let $\alpha$ be a multicurve consisting of $g+1$ non-separating curves $a_1, \ldots, a_{g+1}$ such that $S \setminus \alpha$ is a pair of genus zero subsurfaces, $W^+$ and $W^-$, where $W^+$ contains $\Delta_1,\dots,\Delta_k$ and $W^-$ contains $\Delta_{k+1},\dots, \Delta_b$; see Figure \ref{fig:dis_witness}. 
Orient each curve of $\alpha$ such that $W^+$ is to the left. 
    
    \begin{figure}[h]
    \centering
    \def\svgwidth{4in}
\begingroup%
  \makeatletter%
  \providecommand\color[2][]{%
    \errmessage{(Inkscape) Color is used for the text in Inkscape, but the package 'color.sty' is not loaded}%
    \renewcommand\color[2][]{}%
  }%
  \providecommand\transparent[1]{%
    \errmessage{(Inkscape) Transparency is used (non-zero) for the text in Inkscape, but the package 'transparent.sty' is not loaded}%
    \renewcommand\transparent[1]{}%
  }%
  \providecommand\rotatebox[2]{#2}%
  \newcommand*\fsize{\dimexpr\f@size pt\relax}%
  \newcommand*\lineheight[1]{\fontsize{\fsize}{#1\fsize}\selectfont}%
  \ifx\svgwidth\undefined%
    \setlength{\unitlength}{229.94589248bp}%
    \ifx\svgscale\undefined%
      \relax%
    \else%
      \setlength{\unitlength}{\unitlength * \real{\svgscale}}%
    \fi%
  \else%
    \setlength{\unitlength}{\svgwidth}%
  \fi%
  \global\let\svgwidth\undefined%
  \global\let\svgscale\undefined%
  \makeatother%
  \begin{picture}(1,0.20575752)%
    \lineheight{1}%
    \setlength\tabcolsep{0pt}%
    \put(0,0){\includegraphics[width=\unitlength,page=1]{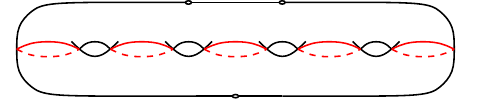}}%
    \put(-0.00433186,0.09395228){\color[rgb]{1,0,0}\makebox(0,0)[lt]{\lineheight{1.25}\smash{\begin{tabular}[t]{l}$\alpha$\end{tabular}}}}%
    \put(0.79887438,0.152562){\makebox(0,0)[lt]{\lineheight{1.25}\smash{\begin{tabular}[t]{l}$W^+$\end{tabular}}}}%
    \put(0.79887976,0.03577642){\makebox(0,0)[lt]{\lineheight{1.25}\smash{\begin{tabular}[t]{l}$W^-$\end{tabular}}}}%
  \end{picture}%
\endgroup%

    \caption{The multicurve $\alpha$ whose complement is a  pair of witnesses for $\Cadm$.}
    \label{fig:dis_witness}
\end{figure}
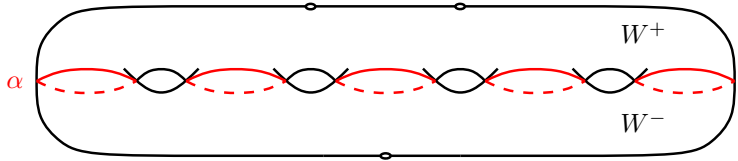

By homological coherence (Lemma \ref{lemma:HJ}.\ref{item:HC}), we have that
for any framing $\psi$ of $S$,
\begin{equation}\label{eqn:HCW+}
\sum_{i=1}^{g+1} x_i + \sum_{j=1}^k \psi(\Delta_j) = 1-g - k
\end{equation}
where $x_i = \psi(a_i)$. 
On the other hand, we know from Lemma \ref{lem:admwitnesses} that $W^+$ contains a (non-peripheral) $\psi$-admissible curve if and only if there is some subset $\mathcal C$ of $\alpha \cup \Delta_1 \cup \ldots \cup \Delta_k$ such that
\begin{equation}\label{eqn:W+adm}
\sum_{c \in \mathcal C} \psi(c) = 1-|\mathcal C|.
\end{equation}
A similar condition tells us if $W^-$ contains any non-peripheral admissible curves.

Now since $g$ of the curves of $\alpha$ are homologically independent, we see that for any $(x_1, \ldots, x_{g+1}) \in \mathbb{Z}^{g+1}$ such that \eqref{eqn:HCW+} holds, there is a framing $\psi$ of $S$ such that $\psi(a_i) = x_i$ for all $i$ and $\psi(\Delta_j) = \phi(\Delta_j)$ for each $j\in \{1,\dots,n\}$ (see \cite[Remark 2.7]{CS}).
Moreover, we can choose $x_i$ not to satisfy \eqref{eqn:W+adm} for any subset $\mathcal C$ of $\partial W^+$ or the corresponding equations for $W^-$ since these all linearly independent from \eqref{eqn:HCW+}.
Thus $W^+$ and $W^-$ are a pair of disjoint witnesses for $\Cadm(S,\psi)$.

Set $K = \sum|\phi(\Delta_j)|$.
The choices in the previous paragraph can all be made explicitly by choosing $x_1, \ldots, x_g$ all to be positive and larger than $2K$ and such that their differences are all larger than $2K$.
Set $x_{g+1}$ to satisfy \eqref{eqn:HCW+}, so it will necessarily be very negative.
Then for any subset $\mathcal C$ of $\alpha \cup \Delta_1, \ldots, \Delta_k$, the left-hand side of \eqref{eqn:W+adm} has magnitude larger than $K$ unless it contains all of $\alpha$.
In this case, any curve separating off (a subset of) the $\Delta_j$ appearing in $W^+$ must have negative winding number, which is in particular not zero. Thus $W^+$ contains no admissible curves, so $W^-$ is a witness.
The argument to show $W^+$ is a witness is completely analogous.

Finally, we note that in the case that $\phi$ is of spin type, we can also choose $\psi$ to have the same $\Arf$ invariant as $\phi$ by stipulating the winding numbers on the completion of $a_1, \ldots, a_g$ to a GSB. 
Theorem \ref{thm:classframed} now provides $f \in \Mod(S)$ such that $\phi = f(\psi)$, and thus $f(W^+)$ and $f(W^-)$ are the desired pair of disjoint witnesses for $\Cadm(S,\phi)$.
\end{proof}

\section{Curve graphs for strata}\label{sec:bdrycx}
In this section we define a number of analogous graphs for (bordifications of) strata. We start by recalling some of the results of \cite{CS} on the relationship between strata, markings, and framed mapping class groups and discussing how the curve complex captures the intersection pattern of the boundary of a bordification of Teichm{\"u}ller space.
Unlike the classical case, it is much more subtle to determine exactly which nodal surfaces can appear in the boundary, leading us to define a number of different graphs that we will eventually prove are all quasi-isometric (Corollary \ref{cor:variants_qi}).

\subsection{Framings and strata}\label{subsec:stratabasics}
A {\em stratum} of abelian differentials is a (quasi-projective) subvariety of the bundle of holomorphic abelian differentials $\Omega \mathcal M_g$ on genus $g$ Riemann surfaces defined by conditioning the number and order of zeros.
More explicitly, given any partition $\sing = (k_1, \ldots, k_n)$ of $2g-2$ into positive integers, we let 
$\Omega \mathcal M_g(\sing) \subset \Omega \mathcal M_g$ denote the stratum parametrizing pairs $(X, \omega)$ where $X$ is a Riemann surface and $\omega$ is a holomorphic 1-form on $X$ with $n$ distinct zeros of orders $k_1, \ldots, k_n$.
Since a holomorphic 1-form is entirely determined (up to global scaling by $\mathbb{C}^*$) by the order and position of its zeros, any stratum can be thought of as a $\mathbb{C}^*$ bundle over a subvariety of $\cM_{g,n}$ (after taking a manifold cover).
In the sequel, we will freely conflate a stratum and its image in $\cM_{g,n}$; we trust this will not cause any confusion.

Let $\Omega \cT_{g,n}(\sing)$ denote the full preimage of the stratum $\Omega \mathcal M_g(\sing)$ inside of $\cT_{g,n}$.
In order to understand its connected components, one needs to understand which mapping classes can be realized inside a stratum, that is, one needs to understand the image of the map 
\[\rho: \pi_1(\cH) \to \pi_1(\cM_{g,n}) \cong \Mod(S_{g,n})\]
of orbifold fundamental groups, where $\cH$ is any stratum component.
When $\cH$ is hyperelliptic, it is not hard to see that the image of $\rho$ is (conjugate to) a hyperelliptic mapping class group \cite{LM, strata1}.
The main theorem of \cite{CS} characterizes the image of $\rho$ for non-hyperelliptic components.

Observe first that a differential $\omega$ has an associated horizontal vector field that does not vanish outside the zeros of $\omega$; we denote this by $1/\omega$.

\begin{theorem}[Theorem A of \cite{CS}]\label{thm:CS}
	Let $\cH$ be a non-hyperelliptic stratum component and suppose that $g \ge 5$. Then the image of $\rho$ is (conjugate to) the framed mapping class group associated to the framing $1/\omega$.
\end{theorem}

We therefore introduce the following notation:

\begin{definition}\label{def:markH}
	Suppose that $\cH$ is a non-hyperelliptic stratum component and let $(X, \omega) \in \cH$. Choose an arbitrary marking $f: S_{g,n} \to X$ and let $\phi$ denote the framing corresponding to the vector field $1/f^*\omega$.
	Then we use $\markH$ to denote the subset of $\Omega \cT_{g,n}(\sing)$ parametrizing those marked differentials $(X', \omega', f')$ such $(X', \omega') \in \cH$ and $1/(f')^*(\omega')$ is isotopic to $\phi$.
\end{definition}

By Theorem \ref{thm:CS}, if $g \ge 5$ then $\cH_\phi$ is just a specified connected component of $\Omega \cT_{g,n}(\sing)$. The reader should think of $\markH$ this way; Definition \ref{def:markH} is written as it is only so that we have something that works for all $g \ge 3$.

The Theorem also reveals a relationship between cylinders and admissible curves.
Integrating $\omega$ induces a singular flat metric on $X$, and the core curve of any embedded Euclidean cylinder has constant slope with respect to the horizontal vector field $1/\omega$, hence winding number 0. Moreover, since the cylinder has nonzero period with respect to a holomorphic 1-form, the core curve must necessarily be non-separating by Stokes's theorem. Thus the core curve is admissible.
Transitivity of the $\FMod(S, \phi)$ action on admissible curves (see Proposition \ref{prop:transitiveadm}) now implies that every admissible curve is realized as a cylinder on some differential in $\markH$ \cite[Corollary 1.2]{CS}.

In Section \ref{sec:transgeom} below, we will use similar transitivity arguments to understand which multicurves can be pinched in the boundary of $\markH$.

\subsection{The curve complex as a nerve}\label{subsec:C(S) as nerve}
Recall that the Deligne--Mumford compactification $\DMcomp$ of the moduli space of Riemann surfaces is obtained by adjoining boundary strata corresponding to (stable) nodal surfaces to $\mathcal M_{g,n}$. Equivalently, it can also be obtained by taking the completion of $\mathcal M_{g,n}$ with respect to the Weil--Petersson metric.
A sequence of surfaces $X_i$ degenerates to the boundary if the (extremal or hyperbolic) length of an essential simple closed curve goes to 0; if $\gamma$ is a topological type of multicurve, then we use $\cM_{g,n}(\gamma)$ to denote the boundary stratum where $\gamma$ is pinched.

One can do a similar thing at the level of Teichm{\"u}ller space.
For any multicurve $\gamma$, let $\cT_{g,n}(\gamma)$ denote the Teichm{\"u}ller space of the open subsurface $S \setminus \gamma$.
The {\em augmented Teichm{\"u}ller space} $\BATS$ is then obtained by adjoining all possible $\cT_{g,n}(\gamma)$ to $\cT_{g,n}$, marking $S \setminus \gamma$ by the subsurface complementary to $\gamma$. 
Equivalently, $\BATS$ is also the Weil--Petersson metric completion of $\cT_{g,n}$.
Points in $\cT_{g,n}(\gamma)$ can be obtained as geometric limits of non-degenerate structures: for example, if $\cT_{g,n} \ni X_i \to X_\infty \in\cT_{g,n}(\gamma)$ then the hyperbolic length of $\gamma$ on $X_i$ goes to $0$, so the $X_i$ develop a long collar that limits to a pair of cusps in $X_\infty$.

We direct the reader to \cite{HK_DM} and its extensive bibliography for a thorough discussion of the history and construction of these spaces.

\begin{remark}
	It is useful (though not quite correct) to think of $\BATS$ as covering $\DMcomp$.
	There is a surjective map $\BATS \to \DMcomp$, which when restricted to any stratum $\cT_{g,n}(\gamma)$ is a covering onto $\mathcal M_{g,n}(\gamma)$, but the overall map is not a covering. This is because $\cT_{g,n}$ is infinitely ramified around the boundary stratum $\cT_{g,n}(\gamma)$ (and likewise $\cT_{g,n}(\gamma)$ is infinitely ramified around its boundary, etc).
\end{remark}

The collar lemma implies that the nerve of the (closures of the) top-dimensional boundary strata of $\BATS$ is exactly given by the usual curve complex $\scrC(S)$ (with vertices given by simple closed curves and simplices given by disjointness).
The 1-skeleton of the barycentric subdivision of the curve complex is the {\em multicurve graph}, which has a vertex for each (simple) multicurve on $S$ and whose edges are given by inclusion: $\gamma$ is connected to $\delta$ if and only if $\gamma \subset \delta$ or $\delta \subset \gamma$. Equivalently, the multicurve graph is the nerve of the (closures of) all boundary strata of $\BATS$.

\subsection{Multi-scale differentials and level splittings}\label{subsec:multiscale}
We now perform a similar construction for (marked) strata of abelian differentials.
Our discussion will be made more complicated by a number of factors, one of which is that some curves on $S$ cannot be pinched by themselves (since any abelian differential is in particular a cohomology class).
In fact, if $\barmarkH \cap \cT_{g,n}(\gamma) \neq \emptyset$ and $\gamma$ is a single simple closed curve, then it must either be admissible or separating. See Section \ref{subsec:divisorial} just below.

\begin{example}[Pinching a multicurve but not its components]
Consider the surface shown in Figure \ref{fig:hompinch}. In this example, curves $\alpha$ and $\beta$ are homologous and so their periods must be equal. Crushing the right-hand torus to have 0 area degenerates into the boundary stratum $\barmarkH \cap \cT_{g,n}(\alpha \cup \beta)$. However, it is impossible to pinch either $\alpha$ or $\beta$ individually while remaining in $\barmarkH$.
\end{example}

\begin{figure}
    \centering
    \def\svgwidth{0.75\linewidth}
    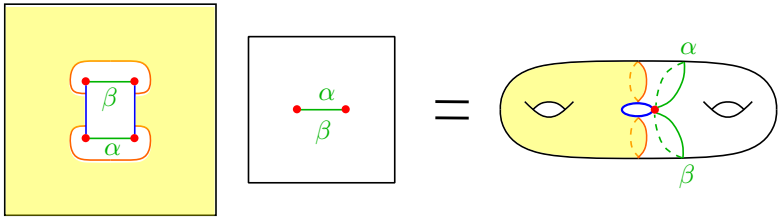 
    \caption{A 2-level splitting consisting of two homologous curves. Neither $\alpha$ nor $\beta$ is admissible, and neither defines a level splitting by itself, so by Theorem \ref{thm:bdry are adm + level} below, $\barmarkH$ does not meet $\cT_{g,n}(\alpha)$ or $\cT_{g,n}(\beta)$.}
    \label{fig:hompinch}
\end{figure}

Specifying $\markH \subset \cT_{g,n}$ as in Definition \ref{def:markH}, let $\barmarkH \subset \BATS$ denote its closure (equivalently, its Weil--Petersson metric completion).
Analogous to the multicurve graph, we now define a graph capturing the pattern of intersections of $\barmarkH$ with the boundary strata of $\BATS$:

\begin{definition}\label{def:bdrycx}
Let $\scrC(\barmarkH)$ be the graph whose vertices are multicurves $\gamma$ such that $\barmarkH \cap \cT_{g,n}(\gamma)$ is nonempty and whose edges are given by inclusion.
\end{definition} 

Exactly which multicurves appear as vertices of $\scrC(\barmarkH)$ is a very intricate question, and is related to subtle properties of a certain compactification of $\cH$.
Let $\overline{\cH}$ be the closure of $\cH$ inside of $\DMcomp$ (without markings). The structure of its boundary is determined by the so-called ``incidence variety compactification'' (IVC) of $\cH$ \cite{IVC}. A point in the IVC consists of a ``level graph'' and a ``twisted differential'' compatible with the level graph; forgetting the differential and remembering only the underlying complex structure yields a surjective map from the IVC onto $\overline{\cH}$ \cite[Corollary 1.4]{IVC}.
It turns out that the IVC is highly singular, and in \cite{multiscale}, the IVC is refined into the moduli space of ``multi-scale differentials'' $\Xi \cH$ which has nicer geometric properties (e.g., its boundary is a normal crossing divisor).
A multi-scale differential is encoded by three pieces of data: an ``enhanced level graph,'' a twisted differential compatible with the level graph and the enhancement, and a ``prong matching.''

We will not give precise definitions of these compactifications here, and direct the reader to the original papers (especially Section 5.1 of \cite{multiscale} and Section 3 of \cite{CMZEuler}).
Instead, we record some of the relevant combinatorial data using our terminology of multicurves and winding numbers. We keep the numbering conventions of \cite{multiscale}.

\begin{definition}
Let $S=S_{g,n}$ and let $\phi$ be a framing of $S$.
An {\em $N$-level splitting} is an oriented multicurve $\vec{\beta}$ together with a partition of $S \setminus \beta$ into (nonempty, but possibly disconnected) subsurfaces $Y_0, \ldots, Y_{-N+1}$ such that:
\begin{itemize}
\item The winding number $\phi(b)$ is negative for every curve $b \subset \vec{\beta}$.
\item Let $b$ be a curve of $\vec{\beta}$. Then if the subsurfaces it sees on its left and right are $Y_i$ and $Y_j$, respectively, then $i > j$.

\end{itemize}
A multicurve $\beta$ is an {\em $N$-level multicurve} if $\beta$ can be oriented and $S \setminus \beta$ can be partitioned to yield an level splitting with $N$ levels.
\end{definition}

\begin{figure}[h]
    \centering
    \def\svgwidth{0.9\linewidth}
    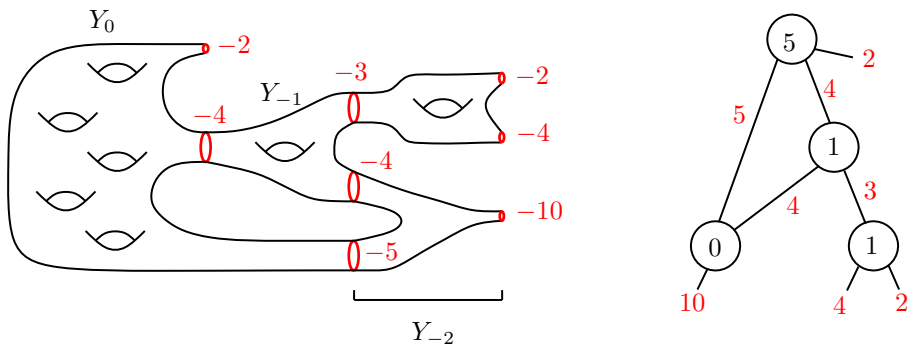
    \caption{A 3-level multicurve and corresponding enhanced level graph.}
    \label{fig:level}
\end{figure}

\begin{remark}
Homological coherence implies that a 2-level multicurve determines a unique level splitting (because there are only two options for the partition of $S\setminus \beta$, and one option does not satisfy homological coherence).
However, as shown in \cite[Examples 3.3 and 3.4]{IVC}, this is not true for general $N$-level multicurves as soon as $N \ge 3$. 
Moreover, it is possible that an $N$-level multicurve is compatible with $N'$-level splittings.
\end{remark}

Comparing this to \cite{multiscale}, a level splitting records slightly different information than an enhanced level graph/enhanced multicurve without horizontal edges.
The dual graph to a splitting (together with the partition of $S \setminus \beta$) is a level graph, and the winding numbers of the curves correspond to the enhancement, i.e., the orders of the zero and pole on each side of the node.
In particular, for each oriented curve $b \subset \vec{\beta}$ corresponding to an edge of the dual graph with enhancement $\kappa$ (so a zero $z$ of order $o(z)=\kappa-1$ and a pole $p$ of order $o(p)=-\kappa - 1$), we have
\[\phi(b) = - \kappa = -1 - o(z) = 1 + o(p).\]
Thus every $N$-level splitting gives rise to an enhanced level graph.

However, a single enhanced level graph may be compatible with multiple level splittings, as the splitting enforces the Arf invariants of the components of $\phi|_{S \setminus \beta}$ (even up to the action of $\Stab_{\Mod(S)}(\beta)$) and the level graph does not.
This is related to the fact that the $\Mod(S)$ orbit of multicurves are generally larger than the $\FMod(S, \phi)$ orbits (even controlling for winding numbers).
Compare Proposition \ref{prop:2 level trans} below.
\medskip

When $\phi$ has holomorphic type, every boundary component and peripheral curve of the top (that is, $0^{th}$) level of a level splitting has negative winding number.
Homological coherence (Lemma \ref{lemma:HJ}.\ref{item:HC})
then implies that each of these top component must have positive genus.
This corresponds to the fact that the top level of a multi-scale differential in the boundary of a holomorphic stratum must itself be holomorphic, hence is supported on a surface with genus.
We record this for later use:

\begin{fact}\label{fact:toplevel_genus}
Let $\vec{\beta}, (Y_0, \ldots, Y_{-N+1})$ be a level splitting of a framed surface of holomorphic type.
Then each component of $Y_0$ has positive genus.
\end{fact}

A corollary of the description of the moduli space of multi-scale differentials in \cite{multiscale} is the following statement, which gives an important necessary condition for which multicurves can be pinched in $\barmarkH$:

\begin{theorem}\label{thm:bdry are adm + level}
If $\barmarkH \cap \cT_{g,n}(\gamma)$ is nonempty, then $\gamma$ is the union of an admissible multicurve $\alpha$ and a disjoint $N$-level multicurve $\beta$.
\end{theorem}

\cite{MUW} gives a sufficient condition for a topological type of multicurve to appear in the boundary of a stratum.
Using results from the literature, one can refine this result to stratum {\em components} $\overline{\cH}$. In particular, using \cite{Wongspin} one can determine exactly when the Arf invariants of subsurfaces enforce the total Arf invariant, and the main result of \cite{ChenFaraco} implies that the global residue condition does not impose any further restrictions.

To further upgrade this to a result for $\barmarkH$, one would also need to establish very strong transitivity results for the action of the framed mapping class group (for example, one needs transitivity on all admissible multicurves of the same topological type, not just pairs).
We were unable to achieve this level of generality, and so instead focus our attention on the ``largest'' boundary strata. For coarse-geometric questions, this distinction will be irrelevant.

\subsection{Divisorial multicurves}\label{subsec:divisorial}
As mentioned above, the map from the space of multiscale differentials $\Xi \cH$ to the closure $\overline{\cH} \subset \DMcomp$ is highly singular. 
All the same, because the boundary of $\Xi\cH$ is a normal crossing divisor, it gives us a good notion of what the largest boundary strata are.

Irreducible components of the boundary divisor of the moduli space $\Xi \cH$ of multiscale differentials correspond to 1-level graphs with a single horizontal edge (i.e., admissible curves) and certain 2-level graphs with only vertical edges (i.e., 2-level multicurves) \cite{multiscale}.
We therefore make the following definition:

\begin{definition}\label{def:div}
A multicurve $\gamma$ is called {\em divisorial} for $\cH_\phi$ if $\barmarkH \cap \cT_{g,n}(\gamma)$ is nonempty and $\gamma$ is either a single admissible curve or a 2-level multicurve.
\end{definition}

It is still fairly complicated to identify exactly which 2-level multicurves are divisorial (compare the discussion at the end of the previous subsection as well as Proposition \ref{prop:2 level trans} below).
As a first example, if $\beta$ is divisorial and $Y_0(\beta)$ has a genus 1 component $U$, then in the associated boundary component of $\Xi \cH$ the surface $U$ is equipped with a holomorphic abelian differential. Thus $\Arf_{1}(\phi|_U)$ must be $0$ (Remark \ref{rmk:holtype vs holdiff}).

All the same, we can build a graph that records only the intersections of (multicurves corresponding to) strata of $\barmarkH$ coming from boundary divisors of $\Xi\cH$. 

\begin{definition}
Set $\CRE$ to be the graph with vertices given by divisorial multicurves $\gamma$, and with an edge between $\gamma$ and $\delta$ if and only if 
$\barmarkH \cap \mathcal T_{g,n}({\gamma \cup \delta})$ is nonempty.
\end{definition}

Observe that $\CRE$ is to the curve graph as $\scrC(\barmarkH)$ is to the multicurve graph.
In particular, its subdivision is a subgraph of $\mathscr{C}(\barmarkH)$ by definition, and since every boundary stratum of $\Xi\cH$ is an intersection of boundary divisors, this subgraph is coarsely dense.

Unfortunately, while it is simpler than $\mathscr{C}(\barmarkH)$, the edges of $\CRE$ are still defined in terms of intersections of boundary strata. This is a subtle question even for disjoint 2-level splittings, so we define one further, simpler curve graph that is more amenable to the HHS techniques from the previous sections.
Eventually, we will show that all of the graphs we have defined are quasi-isometric (Corollary \ref{cor:variants_qi}).

\begin{definition}
Set $\CRV$ to have the same vertex set as $\CRE$, with an edge between $\gamma$ and $\delta$ if and only if the two multicurves are disjoint (but are allowed to share components).
\end{definition}

The graphs $\CRE$ and $\CRV$ are indeed different.
We thank Martin M{\"o}ller for first bringing this phenomenon to our attention.

\begin{example}[Pinching curves but not their union]
Suppose that $S$ has a single puncture and let $\Delta$ denote a curve encircling that puncture.
Let $c$ and $d$ be separating curves on $S$ such that $(\Delta, c, d)$ bounds a pair of pants $P$.
Then $Y_i(c)$ and $Y_i(d)$ both have genus for $i = 0, -1$, and by the main Theorem of \cite{MUW} (or explicit construction), one sees that $\barmarkH$ meets both $\cT_{g,n}(c)$ and $\cT_{g,n}(d)$. 

However, if $\barmarkH$ were to meet $\cT_{g,n}(c \cup d)$, then on any multiscale differential corresponding to this boundary stratum the pair of pants $P$ would be equipped with a meromorphic differential with a single zero and two poles.
Stokes' theorem (more generally, the global residue condition \cite[\S2.4 item (4)]{multiscale}) would then imply that the residues at each pole would be 0, but there is no meromorphic differential on $\widehat{\mathbb{C}}$ with a single zero and two poles of zero residue.
\end{example}

The inclusion gives a 1-Lipschitz map from $\CRE$ to $\CRV$; below, we show that this map actually extends to $\mathscr{C}(\barmarkH)$.

\begin{lemma}\label{lem:Lip from C to E}
There is a coarsely Lipschitz map
\[\xi:\mathscr{C}(\barmarkH) \to \CRV\]
that coarsely agrees with the inclusion $\CRE \hookrightarrow \CRV$.
\end{lemma}
\begin{proof}
Any boundary stratum of $\Xi \cH$ is an intersection of boundary divisors; exactly which divisors can be recovered from the ``undegenerations'' of the associated enhanced level graph \cite[Definition 5.1]{multiscale}.
The precise details of the situation will not be important to us; all we need is the following:

\begin{fact}
Given any boundary stratum of $\Xi \cH$ in which $\gamma$ is pinched, all of its undegenerations correspond to pinching sub-multicurves of $\gamma$.    
\end{fact}

We now define the desired coarsely Lipschitz map 
\[\xi: \scrC(\barmarkH) \to 2^{\CRV}\]
by sending a multicurve $\gamma$ to the set of multicurves pinched in any divisorial undegeneration of any boundary stratum corresponding to $\gamma$.\footnote{
This is one place where our viewpoint of taking level multicurves, not level splittings, makes the discussion more complicated. Which undegenerations occur, and which boundary divisors intersect, depend not just on the multicurve but also on the level structure.}
The image of any vertex of $\scrC(\barmarkH)$ lies in a clique since all of the undegenerations are disjoint, and since edges of $\scrC(\barmarkH)$ are given by inclusion, it follows that if $\gamma \subset \delta$ then $\xi(\gamma) \subset \xi(\delta)$. Thus $\xi$ is coarsely (2-)Lipschitz.

To see that $\xi$ coarsely agrees with the inclusion, we simply observe that it agrees with the inclusion on vertices of $\CRE$, and that the edges of $\CRE$ get mapped to sets containing both endpoints.
\end{proof}

\section{From transitivity to geometry}\label{sec:transgeom}
All of the graphs defined in the previous section carry a natural action of the framed mapping class group. Throughout the section, fix a non-hyperelliptic stratum component $\cH \subset \cM_{g,n}$ (necessarily with $g \ge 3$) and let $\markH$ be as in Definition \ref{def:markH}. Set $S = S_{g,n}$.
Then if $\barmarkH \cap \cT_{g,n}(\gamma) \neq \emptyset$, we have for any $f \in \FMod(S, \phi)$,
\[
f(\barmarkH \cap \cT_{g,n}(\gamma)) 
= \barmarkH \cap \cT_{g,n}(f(\gamma))
\neq \emptyset.
\]
In this section, we analyze this action in more detail and use certain transitivity properties to relate the geometries of $\scrC(\barmarkH)$, $\CRE$, and $\CRV$ to each other, to a hierarchically hyperbolic model, and to the admissible curve graph $\Cadm(S, \phi)$. This will complete the proof of our main Theorem \ref{mainthm:bdrycx}.

As a first example of this technique, let us prove the following:

\begin{lemma}\label{lem:CRE/V contain Cadm}
Both $\CRE$ and $\CRV$ contain $\Cadm(S, \phi)$.
\end{lemma}
\begin{proof}
It suffices to prove the statement for $\CRE$ as it is a subgraph of $\CRV$.
Since every admissible curve is divisorial, $\CRE$ contains the vertices of $\Cadm(S, \phi)$, so it remains to show that it also contains the edges.
By Proposition \ref{prop:transitiveadm}, the framed mapping class group $\FMod(S, \phi)$ acts transitively on pairs of admissible curves of the same topological type.
Thus, it only remains to show that $\barmarkH$ meets {\em some} boundary stratum $\cT_{g,n}(\alpha)$ for each topological type $\alpha$ of pair of admissible curves.

One can do this by explicit construction, one possibility of which we sketch below.
The restriction of $\phi$ to $S \setminus \alpha$ is a framing with four boundary components of winding number $0$.
By holomorphicity of $\phi$, each component of $S \setminus \alpha$ either has positive genus or each peripheral curve on that component has winding number $-1$. 
Pick meromorphic differentials on the components of $S \setminus \alpha$ inducing the same framing and with simple poles corresponding to $\alpha$, all of the same residue (this can be done because strata of meromorphic differentials on surfaces of genus $\ge 1$ with simple poles are always nonempty \cite{Boissy}, and the genus 0 case corresponds to adding free marked points on a cylinder). 
Cutting the infinite cylinders and gluing them together along $\alpha$ yields a holomorphic differential in the correct stratum; applying the (unframed) mapping class group then allows us to ensure that it actually lies in $\markH$. Degenerating these cylinders by letting their heights go to $\infty$ then produces a path in $\markH$ to $\cT_{g,n}(\alpha)$.

The only thing one might worry about is matching the Arf invariants of the subsurfaces to ensure that the plumbed surface has the correct Arf invariant: this turns out not to be an issue for the following reason.
If $g \ge 4$ then each stratum of meromorphic differentials in genus $\ge 2$ has components of both spin parities \cite[Theorem 1.2]{Boissy}, so by choosing the appropriate Arf invariants on pieces we can ensure that the plumbed surface has the appropriate Arf invariant.
In the special case that $g =3$, there is a unique component of meromorphic differentials on a genus 1 surface with two simple poles and a single zero of order 2, and so the plumbed surface is forced to have odd Arf invariant. Fortunately, this only happens in the stratum $\Omega\cM_3(2,2)$, which has a unique non-hyperelliptic component of odd Arf invariant \cite[Theorem 2]{KZ}.
\end{proof}

\subsection{The action on 2-level multicurves}

2-level multicurves can have many different topological types, so $\FMod(S, \phi)$ will certainly not act transitively on them. However, even controlling for topological type and winding numbers, the Arf invariants of subsurfaces present additional invariants of the $\FMod(S, \phi)$ orbit.
We show below that these are the only obstructions to transitivity.

While we will not use this in the sequel, note that the following statement is true for {\em all} 2-level splittings of a surface of holomorphic type, not just divisorial ones.

\begin{proposition}\label{prop:2 level trans}
Let $\phi$ be a framing of holomorphic type on a surface $S$ of genus at least 3.
Let $\beta$ be any 2-level multicurve.
Then a multicurve $\beta'$ is in the $\FMod(S, \phi)$ orbit of $\beta$ if and only if there exists an $h \in \Mod(S)$ such that:
\begin{enumerate}
    \item $h(\beta) = \beta'$.
    \item $\phi(b) = \phi(h(b))$ for every curve $b \in \beta$.
    \item For each component $U$ of $S \setminus \beta$ of genus at least 2 such that $\phi|_U$ is of spin type,
    \[\Arf(\phi|_U) = \Arf( \phi|_{h(U)}),\]
    and similarly, for any complementary component $U$ of genus 1,
    \[\Arf_1(\phi|_U) = \Arf_1( \phi|_{h(U)}).\]
\end{enumerate}
\end{proposition}
\begin{proof}
We are given $h \in \Mod(S)$ taking $\beta$ to $\beta'$; our goal is to find an element in the (orientation-preserving, component-wise) stabilizer of $\beta'$ such that its composition with $h$ preserves the winding numbers of a GSB for $S$.
We will construct this element and the associated GSB in steps, starting from the bottom and working up.
The reader is invited to compare with the discussion of ``perturbed period coordinates,'' especially Figure 5, in \cite{multiscale}.
Throughout the proof, given any curve of $\beta$ or component $U$ of $S \setminus \beta$, we will add a prime to denote its image under $h$, i.e., $U' := h(U)$.

\medskip{\noindent \bf Bottom level:}
Choose a GSB $\mathcal{B}_U$ on each component $U$ of $Y_{-1}(\beta)$. Hypothesis (3) allows us to apply Lemma \ref{lem:existsGSB} to choose a GSB $\mathcal B_{U'}$ for $U'$ with the same set of winding numbers as appear in $\mathcal B_{U}$.
Using the classical change-of-coordinates principle, we can find some element $f_{U'} \in \Mod(U')$ (which we can then think of as living in $\Mod(S)$ via inclusion) that takes $h(\mathcal{B}_U)$ to $\mathcal B_{U'}$.
Set 
\[f_{bot} =\prod_{U' \subset Y_{-1}(\beta')} f_{U'} \circ h;\]
by construction it takes $\beta$ to $\beta'$ and preserves the winding numbers of a GSB for $Y_{-1}(\beta)$.

\medskip{\noindent \bf Level passage:}
For this and the next step, for each component $U$ of the top level $Y_0(\beta)$, pick a subsurface $V_U \subset U$ with full genus and a single boundary component.

Pick a maximally homologically independent subset $\mathsf{b} = (b_1, \ldots, b_h)$ of $\beta$ and extend 
\[\bigcup_{U \subset Y_{-1}(\beta)} \mathcal B_{U} \cup \mathsf b\]
to a GSB for the complement of all of the $V_U$. Let $\mathsf c = (c_1, \ldots, c_h)$ denote the resulting set of curves symplectically dual to those of $\mathsf b$.
The element $f_{bot}$ will most likely not preserve the winding numbers of $\mathsf c$, but we can rectify this using elements supported entirely on $Y_0(\beta')$.

Consider first $c_1$; the dual curve $b_1$ is a curve of the 2-level splitting $\beta$, and we use $U_1$ to denote the component of $Y_0(\beta)$ adjacent to $b_1$.
Since $\phi$ is of holomorphic type, $U_1$ has genus, hence so does the full-genus subsurface $V_1 := V_{U_1}$.
Set $\mathsf c' = f_{bot}(\mathsf c)$ and likewise for its components. 
Pick some admissible curve $a_1' \subset f_{bot}(V_1)$ and let $d_1'$ denote the connect sum of $b_1'$ with $a_1'$ along some arc contained in $U_1'$. 
Since $a_1'$ is disjoint from $\mathsf c'$ and $b_1'$ only meets $c_1'$, we see that the algebraic intersection number of $d_1'$ with each $c_j'$ is 0 unless $j=1$, in which case it is exactly $1$. See Figure
\ref{fig:fixlevels}.

\begin{figure}
    \centering
    \def\svgwidth{0.6\linewidth}
    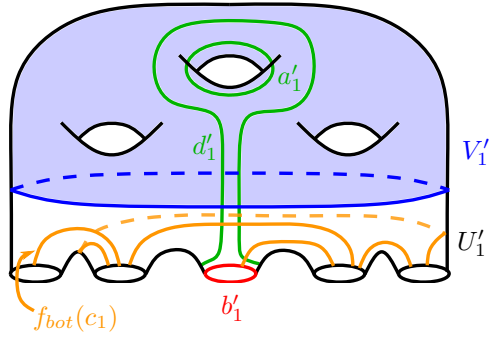
    \caption{The multitwist needed to fix the winding number of $f_{bot}(c_1)$. Note that it may not always be possible to choose the curve $d_1'$ to be disjoint from the other curves of $f_{bot}(\mathsf c)$.}
    \label{fig:fixlevels}
\end{figure}

By homological coherence (Lemma \ref{lemma:HJ}.\ref{item:HC}), 
\[\phi(b_{1}') - \phi(d_1') = -1\]
when appropriately oriented. Thus, if we set 
\[f_1 := \left(T_{b_{1}'}^\pm T_{d_1'}^\pm \right)
^{\phi(c_{1}') - \phi(c_{1})}\]
for an appropriate choice of signs, then by twist-linearity (Lemma \ref{lemma:HJ}.\ref{item:TL}) we see that $f_1(c_1')$ has the same winding number as $c_1$ and that $f_1$ preserves the winding numbers of all other $c_j'$.

We now repeat the above procedure but with 
$f_1 \circ f_{bot}$ instead of $f_{bot}$.
More precisely, set $V_2 \subset U_2$ to be the full-genus subsurface of the top-level component of $S \setminus \beta$ adjacent to $b_2$.\footnote{The component $U_2$ (and subsurface $V_2$) may be the same as $U_1$ (and $V_1$).}
There is an admissible $a_2' \subset f_1 f_{bot}(V_2)$, and taking the connect sum of $b_2'$ with this curve yields some $d_2'$ whose algebraic intersection with each curve of $f_1 f_{bot}(\mathsf c)$ is 0 except for $f_1 f_{bot}(c_2)$. Taking an appropriate multitwist in $b_2'$ and $d_2'$ yields some $f_2$ supported on $U_2'$ such that $f_2 f_1 f_{bot}$ preserves the winding numbers of both $c_1$ and $c_2$.

Iterating, we get a sequence of mapping classes $f_1, \ldots, f_{h}$ all supported on $Y_{0}(\beta')$ such that the composite
\[f_{mid} := f_h\circ \ldots \circ f_1 \circ f_{bot}\]
takes $\beta$ to $\beta'$ and preserves the winding numbers of the curves of a GSB for the complement of the full-genus subsurfaces $V$ of the top-level components $U \subset Y_0(\beta)$. 

\medskip{\noindent \bf Top level:}
To finish, we can simply use the action of $\Mod(f_{mid}(V))$ to amend the winding numbers of the remaining curves as we did for the bottom level.

Pick a GSB $\mathcal{B}_V$ for each such $V$. Then by Lemma \ref{lem:fullgenusArf} and hypothesis (3), we have that 
\[\Arf(\phi|_{V})= \Arf(\phi|_{U}) = \Arf(\phi|_{U'}) = \Arf(\phi|_{f_{mid}(V)})\]
when defined. If $U$ is of genus $1$, the same thing holds for the genus 1 Arf invariant (note that this requires the fact that $\phi|_U$ is of holomorphic type!).
Lemma \ref{lem:existsGSB} then implies that
$f_{mid}(V)$ admits a GSB with the same winding numbers as $\mathcal{B}_V$ and we pick an element $f_V \in \Mod(f_{mid}(V))$ taking $f_{mid}(\mathcal{B}_V)$ to this GSB.

Finally, we observe that the mapping class $\prod_{V} f_V \circ f_{mid}$ takes $\beta$ to $\beta'$ and preserves the winding numbers of the following GSB for $S$:
\[
\bigcup_{U \subset Y_{-1}(\beta)} \mathcal B_{U}
\cup \mathsf b \cup \mathsf c
\cup \bigcup_{U \subset Y_{0}(\beta)} \mathcal B_{V}.
\]
We have therefore constructed the desired framed mapping class.
\end{proof}

\subsection{Pinching admissible curves}

Using the same ideas as Proposition \ref{prop:2 level trans}, we show that every 2-level splitting is connected to some admissible curve in $\CRE$.

\begin{proposition}\label{prop:realize_adm}
Suppose that $\beta$ is a divisorial 2-level multicurve for $\barmarkH$, i.e., $\barmarkH$ meets $\cT_{g,n}(\beta)$.
Then for any admissible $a \subset Y_0(\beta)$, we have that $\barmarkH$ meets $\cT_{g,n}(\beta \cup a)$.
\end{proposition}
\begin{proof}
We first show that one can further pinch {\em some} admissible curve in $Y_0(\beta)$.
The restriction of a multiscale differential lying over $\barmarkH \cap \cT_{g,n}(\beta)$ is holomorphic on $Y_0(\beta)$. Every holomorphic differential contains an embedded nonsingular cylinder \cite{Masur_existscyl} whose core curve $a'$ is necessarily admissible,
and one can degenerate into $\barmarkH \cap \cT_{g,n}(\beta \cup a')$ by sending the height of this cylinder to $\infty$.
\medskip

Thus, it suffices to show that the stabilizer of $\beta$ in $\FMod(S, \phi)$ acts transitively on the set of admissible curves contained in each component of $Y_0(\beta)$.
Let $a$ and $a'$ be different admissible curves contained in the same component $U$ of $Y_0(\beta)$ and suppose (postcomposing by an element of $\Mod(U)$ as necessary) that the element $h$ taking $\beta$ to $\beta'$ also takes $a$ to $a'$.
The proof of Proposition \ref{prop:2 level trans} then proceeds by upgrading $h$ into a framed mapping class; we show that each step, this can done be done in a way that preserves $a'$ and so the composite element still takes $a$ to $a'$.

\medskip{\noindent \bf Bottom level:} The element $f_{bot}$ differs from $h$ by an element supported entirely on the bottom level $Y_{-1}(\beta)$, so still takes $a$ to $a'$.

\medskip{\noindent \bf Level passage:}
Pick the full-genus subsurface $V \subset U$ to contain $a'$, so $f_{bot}(V)$ contains $a'$.
Each element $f_j$ is constructed by taking a multitwist disjoint from some choice of admissible curve.
So long as we pick $a'$ to be this admissible base curve each time that $U$ is the relevant subsurface in the iteration, then the resulting multitwist $f_j$ will preserve $a'$ and the new subsurface $f_j \cdots f_1 f_{bot}(V)$ will still contain $a'$.
Thus $f_{mid}$ must also take $a$ to $a'$.

\medskip{\noindent \bf Top level:}
The last step of the construction takes a fixed GSB of $V$ to a GSB of $f_{mid}(V)$ with the same winding numbers.
We now just make sure to take $a$ as an element of the GSB of $V$ and take $a'$ to be the corresponding element of the GSB of $f_{mid}(V)$.
The fact that we can extend $a'$ to a GSB with the appropriate winding numbers is immediate in the genus 1 case, as divisoriality implies that $\Arf_1(\phi_V)=0$, so every nonseparating curve is admissible.
If $g(V) \ge 2$, it follows from Corollary \ref{cor:specifytransverse} that we can choose a curve transverse to $a'$ with the desired winding number, then apply Lemma \ref{lem:existsGSB} to the complement in $V$ of the subsurface filled by $a$ and this curve (note that if $g(V)=2$, then the boundaries of this subsurface have winding numbers $-3$ and $+1$, so its genus 1 Arf invariant is either $1$ or $2$ and there are no more restrictions than already appear from fixing the Arf invariant of $V$).
\end{proof}

Combining this with Lemma \ref{lem:CRE/V contain Cadm} allows us to quickly conclude connectivity.

\begin{corollary}\label{cor:variantsconn}
The graphs $\mathscr{C}(\barmarkH)$, $\CRE$, and $\CRV$ are connected.
\end{corollary}
\begin{proof}
Since $\CRE$ is a subgraph of $\CRV$ (and its subdivision is a subgraph of $\mathscr{C}(\barmarkH)$), it suffices to prove this for $\CRE$.
By Lemma \ref{lem:CRE/V contain Cadm}, $\CRE$ contains the admissible curve graph $\Cadm(S, \phi)$, and as shown in Lemma \ref{lem:Cadmconn}, $\Cadm(S, \phi)$ is connected.
Every vertex of $\CRE$ that is not an admissible curve is a 2-level multicurve $\beta$.
There is some admissible curve $a$ contained inside of each component of $Y_0(\beta)$ (Fact \ref{fact:toplevel_genus} and Lemma \ref{lem:genusadm}) and so Proposition \ref{prop:realize_adm} implies that $\barmarkH$ meets $\cT_{g,n}(\beta \cup a)$.
Thus $\beta$ and $a$ are connected in $\CRE$.
Since we have connected every vertex of $\CRE$ to the connected graph $\Cadm(S, \phi)$, we conclude that $\CRE$ is connected.
\end{proof}

Pushing this line of reasoning slightly further, we also get the following:

\begin{corollary}\label{cor:variants_qi}
The graphs $\mathscr{C}(\barmarkH)$, $\CRE$, and $\CRV$ are quasi-isometric.
\end{corollary}
\begin{proof}
Observe that via the inclusion (of its subdivision) and Lemma \ref{lem:Lip from C to E}, we have already built coarsely Lipschitz maps
\[\CRE \to \mathscr{C}(\barmarkH) \to \CRV\] such that the final map $\CRE \to \CRV$ coarsely agrees with the inclusion.
Thus, it remains to show that the inclusion $\CRE \hookrightarrow \CRV$ has a coarse inverse. 

Define $\zeta\colon \CRV \to 2^{\CRE}$ to be the identity on the vertices of $\CRV$ and to send each edge of $\CRV$ to the pair of vertices in $\CRE$ that are its endpoints. If $\zeta$ is coarsely Lipschitz, then it is necessarily a coarse inverse of the inclusion $\CRE \hookrightarrow \CRV$.

To that end, consider any edge of $\CRV$ connecting disjoint multicurves $\gamma$ and $\delta$ that correspond to divisorial boundary strata. We just need to show a bound on the length of a path from $\gamma$ to $\delta$ in $\CRE$.
If both $\gamma$ and $\delta$ are single admissible curves, then Lemma \ref{lem:CRE/V contain Cadm} implies they are connected in $\CRE$.
Now suppose $\gamma$ is admissible and $\delta$ is a 2-level curve. As in the proof of Corollary \ref{cor:variantsconn}, 
each component of $Y_0(\delta)$ has genus and if $\gamma \subset Y_0(\delta)$ then Proposition \ref{prop:realize_adm} implies that $\gamma$ and $\delta$ are connected by an edge of $\CRE$.
Otherwise, $\gamma \subset Y_{-1}(\delta)$ and in particular it is disjoint from $Y_0(\delta)$.
By Lemma \ref{lem:genusadm}, there is an admissible curve $a$ on $Y_0(\delta)$. Applying Proposition \ref{prop:realize_adm} again we see that $\gamma$ is connected to $a$ which is connected to $\delta$ (by Lemma \ref{lem:CRE/V contain Cadm}).

Finally, suppose that both $\gamma$ and $\delta$ are 2-level curves.
If any components of $Y_0(\gamma)$ and $Y_0(\delta)$ are nested, then since they both have genus their intersection does. In particular by Lemma \ref{lem:genusadm} there is some admissible curve $a$ disjoint from both $\gamma$ and $\delta$. We may then invoke Proposition \ref{prop:realize_adm} again to connect $\gamma$ to $\delta$ in $\CRE$ through $a$.
Otherwise, $Y_0(\gamma)$ and $Y_0(\delta)$ are disjoint. In this case we choose admissible curves $a_\gamma$ and $a_\delta$ inside $Y_0(\gamma)$ and $Y_0(\delta)$, respectively.
The previous two paragraphs then imply that $(\gamma, a_\gamma, a_\delta,\delta)$ is an edge path in $\CRE$. Thus, collecting cases, we have shown that each edge of $\CRV$ is sent to a set of diameter at most $4$, hence $\zeta$ is coarsely Lipschitz.
\end{proof}

\subsection{A quasi-isometry with the model}\label{subsec:geom_CRECRV}
We now build off our work showing $\Cadm(S, \phi)$ is hierarchically hyperbolic to prove that $\CRV$ (hence $\CRE$ and $\mathscr{C}(\barmarkH)$) are as well.

As in the case of the admissible curve graph, we establish the hierarchical hyperbolicity of $\CRV$ by showing it is quasi-isometric to a model graph constructed from its witnesses.
Because $\wit(\CRV)$ is a proper subset of $\wit(\Cadm(S, \phi))$, our proof for $\CRV$ will actually rely on the proof for $\Cadm(S, \phi)$. To describe this setup, we need the following notation:

\begin{itemize}
    \item Set $\xi= \xi(S)$, the cardinality of the largest set of disjoint curves on $S = S_{g,n}$.
    \item Let $\mathcal{D}$ be the set of divisorial 2-level splittings for $\barmarkH$ (these are exactly the vertices of $\CRV$ that are not in $\Cadm(S, \phi)$).
    \item Let $\fS = \wit(\Cadm(S, \phi))$ and $\ofS = \wit(\CRV)$. By definition,
\[\overline{\fS} =  \fS \setminus \{ W \in \fS : \exists \delta \in \cD \text{ with } \delta \cap W = \emptyset\}.\]
\item Let $\cK = \cK_\fS$ denote the quasi-isometric model for $\Cadm(S, \phi)$ (Definition \ref{definition: ksep}) and let $\overline{\cK}$ denote $\cK_{\ofS}$. 
\end{itemize}

By construction, there are 1-Lipschitz inclusion maps 
\[i \colon \cK \to \ocK \text{ and } \iota\colon \Cadm(S, \phi) \to \CRV.\]
The idea behind our proof that $\ocK$ is quasi-isometric to $\CRV$ is to show that the decreases in distances that happen under $i \colon \cK \to \ocK $ coarsely match the decreases that happen under $\iota\colon \Cadm(S, \phi) \to \CRV$. 

To formalize this idea, we define $$ P(\mu) = \{\alpha \in \cK : \mu \subseteq \alpha\}$$ for any multicurve $\mu$ on $S$.
If $\mu$ is a multicurve such that $S \setminus \mu$ does not contain a subsurface in $\fS$, then $\mu$ is a vertex of $\cK$ and every vertex of $P(\mu)$ is connected to $\mu$ by a path with at most $\xi$ edges (corresponding to removing curves until only $\mu$ is left).
On the other hand, if $S\setminus \mu$ does contain a subsurface in $\fS$, then $P(\mu)$ has infinite diameter; see \cite[Corollary 4.10]{russell_vokes_sep}. 
Hence, if $\mu$ is a vertex of $\ocK$, but not $\cK$, then $P(\mu)$ is an infinite diameter subset of $\cK$ that becomes finite diameter under $i \colon \cK \to \ocK$. If $\ocK$ is to be quasi-isometric to $\CRV$, we would like the image of $P(\mu)$ under the quasi-isometry $\cK \to \Cadm(S, \phi)$ to have uniformly bounded diameter under $\iota \colon \Cadm(S, \phi) \to \CRV$.

Our candidate quasi-isometry $\Theta \colon \ocK 
 \to 2^{\CRV}$ is therefore 
 \[\Theta(\mu) = \iota \circ \Pi \circ \Psi(P(\mu))\]
 where $\Pi \circ \Psi$ is the quasi-isometry from $\cK$ to $\Cadm(S, \phi)$ constructed in Section \ref{sec:Cadm_qi_model}.= 
 
 As suggested above, the main work required to prove $\Theta$ is a quasi-isometry is to show that $\iota \circ \Pi \circ \Psi(P(\mu))$ is a bounded diameter subset of $\CRV$. To achieve this, we need to know that $\CRV$ is obtained from $\Cadm(S, \phi)$ by adding enough divisorial 2-level  multicurves to collapse the image of $P(\mu)$ for each $\mu$ that is in $\ocK$ but not $\cK$. The abundance of 2-level multicurves comes from the following lemma. 

\begin{lemma}\label{lem:lots_of_2-level_curves}
For any $\markH$, there is an $N$, depending only on $S$, such that the following holds. 
Let $W$ be a genus $0$ witness of $\Cadm(S, \phi)$
and let $\beta \in \cD$ be a divisorial 2-level multicurve disjoint from $W$.
Then for any multicurve $\alpha$ on $S \setminus W$, there is an $f \in \FMod(S, \phi)$ such that $f(\beta)$ remains disjoint from $W$ and
$i(\alpha, f(\beta)) \le N.$
\end{lemma}

Note that in particular, $f(\beta) \in \cD$ since $\cD$ is a union of $\FMod(S, \phi)$ orbits.

This is a weaker, framed version of the following standard ``change of coordinates'' lemma.

\begin{lemma}\label{lem:bdd int Mod orbit}
For any surface $Z$, there is an $N_Z$ such that for any multicurves or multiarcs $\alpha$ and $\beta$, there is a $g\in \Mod(Z)$ such that
$i(\alpha, g(\beta)) \le N_Z$.
\end{lemma}

As in Section \ref{sec:Cadm_qi_model}, the reason that we cannot use a similar ``change of coordinates'' argument (even though we have shown the set of divisorial 2-level splittings for $\barmarkH$ is a finite union of $\FMod(S, \phi)$ orbits) is that there are infinitely many $\FMod(S, \phi)$ orbits of witnesses. Instead of making a finite number of arbitrary choices, we will instead need to be more clever and make a infinite number of good ones.

\begin{proof}[Proof of Lemma \ref{lem:lots_of_2-level_curves}]
We first record a number of topological consequences of our hypotheses. Let $Y_0$ and $Y_{-1}$ denote the two levels of the 2-level splitting associated to $\beta$. Then since $W$ is a witness and $(S, \phi)$ is of holomorphic type, we see that
\begin{itemize}
\item $W \subset Y_0$,
\item each component of $Y_{-1}$ has genus 0,
\item $Y_0$ is connected, and
\item $Y_0$ has genus at least $2$.
\end{itemize}
Indeed, contradicting any of the first three statements immediately implies that there is an admissible curve disjoint from $W$ (Lemma \ref{lem:genusadm}).
The last assertion follows similarly: $Y_0$ always has positive genus, and some curve of $\partial W$ is always non-separating on $Y_0$. If the genus of $Y_0$ were equal to $1$ then $\Arf_1(Y_0)=0$ must be 0 by divisoriality (see the discussion right after Definition \ref{def:div}) and hence some boundary curve of our witness $W$ would be admissible, a contradiction.

We observe that since $S \setminus W$ has genus $0$, the winding number of any curve on $S\setminus W$ is determined by how it partitions the curves of $\partial W$ and the punctures of $S$.
Thus, any mapping class supported entirely on $W$ necessarily preserves the winding numbers of the curves of $\beta$.
Applying Lemma \ref{lem:bdd int Mod orbit} (and using the inclusion homomorphism for subsurfaces \cite[Theorem 3.18]{FarbMarg}), we can therefore find some $h(\beta) \in \Mod(S) \cdot \beta$ that has bounded intersection with $\alpha$ and with the same winding numbers as $\beta$.
Note that $h(\beta)$ is a 2-level splitting, as the definition of 2-level splitting depends only on winding numbers. Note also that $h(\beta)$ may not be divisorial.

In the case that $Y_0$ is not of spin type, 
or if $\Arf(hY_0)$ happens to equal $\Arf(Y_0)$, then Proposition \ref{prop:2 level trans} ensures that $\beta$ and $h(\beta)$ are in the same $\FMod(S, \phi)$ orbit, completing the proof.

Otherwise, $\Arf(hY_0) = \Arf(Y_0)+1$.
Our goal is now to amend the Arf invariant of $hY_0$ while introducing a uniformly bounded number of intersections with $\alpha$.
We note first that the topology of the situation at hand forces there to be a curve $w$ of $\partial W$ that is nonseparating on $Y_0$ (hence on $hY_0$) and such that $\phi(w)$ is even.
Indeed, if the winding numbers of all such $w$ were odd, then since $Y_0 \setminus W$ has genus 0 this would imply that $\phi$ is odd on curves spanning a Lagrangian subspace of $H_1(Y_0; \ZZ)$.
In particular, this would imply that each term in the formula for the Arf invariant \eqref{eqn:arfdef} would be $0$, hence $\Arf(Y_0) = 0$. The same would therefore also be true for $\Arf(hY_0)$, but this is a contradiction.

Now every mapping class supported entirely on $W$ can be written as a product of Dehn twists, and every curve on $S \setminus W$ is separating.
Thus, for any curve $c$ on $S$ and any curve $d \subset S \setminus W$, we have that 
\[\hat\iota (hc, d) = \hat\iota (c, d)\]
where $\hat \iota (\cdot, \cdot)$ denotes the algebraic intersection number.
Combining this with twist-linearity, we see that if we factorize
$h = T_{d_1}^{k_1} \cdots T_{d_n}^{k_n}$ where $d_i$ are curves on $S \setminus W$, then
\[\phi(hc) = 
\phi(c) + 
k_1 \hat\iota (c, d_1) \phi(d_1)
+ \ldots +
k_n \hat\iota (c, d_n) \phi(d_n)\]
for any curve $c$ on $S$.

Returning to the situation at hand, since $\Arf(hY_0) = \Arf(Y_0) +1$, the discussion above implies there must be some curve $c \subset Y_0$, part of a GSB for $Y_0$ and symplectically dual to a curve $w \subset \partial W$ of even winding number, and some curve $d \subset S \setminus W$ such that
$\hat\iota (c, d)$ and $\phi(d)$ are both odd.
Moreover, $\hat\iota (hc, d)$ is also odd, and since algebraic intersection number and winding number properties on a genus 0 surface depend only on how a curve separates the surface, Lemma \ref{lem:bdd int Mod orbit} ensures there is a $d'$ on $S \setminus W$ such that
\begin{enumerate}
    \item $\hat\iota (gc, d')$ is odd
    \item $\phi(d')$ is odd
    \item The geometric intersection number of $d'$ with $\alpha$ is uniformly bounded.
\end{enumerate}
Comparing with Formula \eqref{eqn:arfdef}, items (1) and (2) ensure that $T_{d'}hY_0$ has the same Arf invariant (and boundary winding numbers) as $Y_0$, hence Proposition \ref{prop:2 level trans} implies that $\beta$ and $T_{d'}h\beta$ are in the same $\FMod(S, \phi)$ orbit. Since the geometric intersection of $h\beta$ and $\alpha$ was uniformly bounded, item (3) ensures that the geometric intersection of $T_{d'}h\beta$ and $\alpha$ is as well.
\end{proof}

We can now prove $\ocK$ is quasi-isometric to $\CRV$; thus $\CRV$ is hierarchically hyperbolic.

\begin{theorem}\label{thm:bdrycx model qi}
    The map $\Theta \colon \overline{\cK} \to 2^{\CRV}$ is a quasi-isometry. 
\end{theorem}

\begin{proof}
Throughout the proof, we say a quantity is uniform if it depends only on the surface $S$. 
Set $\mathscr{E} := \CRV$ and let $\theta = \iota \circ \Pi \circ \Psi$, so $\Theta(\mu) = \theta(P(\mu))$.

Our proof has three steps. First we prove that $\diam(\Theta(\mu))$ is uniformly bounded for each vertex $\mu \in \overline{\cK}$. Then we show that if $\mu$ and $\nu$ are joined by an edge of $\overline{\cK}$ then $\diam(\Theta(\mu) \cup \Theta (\nu))$ is also uniformly bounded. Together these show that $\Theta$ is coarsely Lipschitz.
Finally, we check that $\Theta$ is a coarse inverse to the inclusion map $\mathscr{E} \to \overline{\cK}$.

\medskip
{\noindent \bf Step 1: vertices have uniform diameter.} 
      If no component of $S \setminus \mu$ is an element of $\fS$, then $\mu \in \cK$ and every vertex of $P(\mu)$ is obtained by adding fewer than $\xi$ curves to $\mu$. Hence $\diam(P(\mu)) \leq 2 \xi$. Since $\theta =\iota \circ \Pi \circ \Psi$ is coarsely Lipschitz, this implies $\theta(P(\mu))=\Theta(\mu)$ is uniformly bounded.      
      Hence, we can assume there exist a component $W$ of $S \setminus \mu$ that is in $\fS$, i.e., is a witness for $\Cadm$.

      For our fixed $\mu \in \overline{\cK}$, let $\alpha \in P(\mu)$. The multicurve $\alpha$ is the union of three distinct sets: $\mu$, $\alpha_W = \alpha \cap W$, and $ \alpha \setminus (\mu \cup \alpha_W)$; see Figure \ref{fig:alpha_W}. Set $\alpha' := \mu \cup \alpha_W$.
      We divide the remainder of our proof into three cases based on the subsurface $W$. 

      \begin{figure}[h]
      \centering
        \def\svgheight{3cm}
\begingroup%
  \makeatletter%
  \providecommand\color[2][]{%
    \errmessage{(Inkscape) Color is used for the text in Inkscape, but the package 'color.sty' is not loaded}%
    \renewcommand\color[2][]{}%
  }%
  \providecommand\transparent[1]{%
    \errmessage{(Inkscape) Transparency is used (non-zero) for the text in Inkscape, but the package 'transparent.sty' is not loaded}%
    \renewcommand\transparent[1]{}%
  }%
  \providecommand\rotatebox[2]{#2}%
  \newcommand*\fsize{\dimexpr\f@size pt\relax}%
  \newcommand*\lineheight[1]{\fontsize{\fsize}{#1\fsize}\selectfont}%
  \ifx\svgwidth\undefined%
    \setlength{\unitlength}{416.51649049bp}%
    \ifx\svgscale\undefined%
      \relax%
    \else%
      \setlength{\unitlength}{\unitlength * \real{\svgscale}}%
    \fi%
  \else%
    \setlength{\unitlength}{\svgwidth}%
  \fi%
  \global\let\svgwidth\undefined%
  \global\let\svgscale\undefined%
  \makeatother%
  \begin{picture}(1,0.15948812)%
    \lineheight{1}%
    \setlength\tabcolsep{0pt}%
    \put(0,0){\includegraphics[width=\unitlength,page=1]{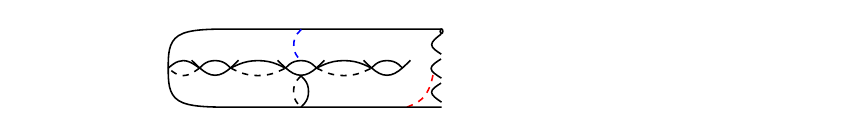}}%
    \put(0.22870162,0.13609118){\color[rgb]{0,0,0}\transparent{0.30000001}\makebox(0,0)[lt]{\lineheight{1.25}\smash{\begin{tabular}[t]{l}$W$\end{tabular}}}}%
    \put(0.35916747,0.04854411){\makebox(0,0)[lt]{\lineheight{1.25}\smash{\begin{tabular}[t]{l}$\mu$\end{tabular}}}}%
    \put(0,0){\includegraphics[width=\unitlength,page=2]{alpha_W_svg-tex.pdf}}%
    \put(0.42512923,0.00679463){\color[rgb]{1,0,0}\makebox(0,0)[lt]{\lineheight{1.25}\smash{\begin{tabular}[t]{l}$\alpha \setminus (\mu \cup \alpha_W)$\end{tabular}}}}%
    \put(0.33985561,0.13759899){\color[rgb]{0,0,1}\makebox(0,0)[lt]{\lineheight{1.25}\smash{\begin{tabular}[t]{l}$\alpha_W$\end{tabular}}}}%
    \put(0,0){\includegraphics[width=\unitlength,page=3]{alpha_W_svg-tex.pdf}}%
    \put(-0.00239148,0.00693596){\color[rgb]{1,1,1}\makebox(0,0)[lt]{\lineheight{1.25}\smash{\begin{tabular}[t]{l}$\alpha \setminus (\mu \cup \alpha_W)$\end{tabular}}}}%
  \end{picture}%
\endgroup%

          \caption{The partition of $\alpha \in P(\mu)$ into $\alpha_W$, $\mu$ and $\alpha \setminus (\mu \cup \alpha_W)$.}
          \label{fig:alpha_W}
      \end{figure}

      \emph{Case 1: $g(W) \geq 1$.}
      Since $\mu \in \overline{\cK}$, then by definition of $\overline{\cK}$ there must exist some 2-level splitting $\delta \in \mathcal{D}$ such that $W$ is disjoint from $\delta$.
      Since $W$ has genus, no component of $S \setminus W$ can be in $\fS$, thus $\alpha'$ is a vertex of $\cK$ that is joined by a path of length at most $\xi$ to $\alpha$.  
      By Lemma \ref{lem:sep_curve_int_multicurve}, there exists a curve $c \subset W$ that cuts off a genus 1 subsurface and such that $c \in \Psi(\alpha')$.
      Thus, there is an admissible curve $a \in \Pi \circ \Psi(\alpha')$ contained in the subsurface $W$. Since $a$ is disjoint from $\delta$, there is an edge of $\mathscr{E}$ from  a vertex of $ \theta(\alpha')$ to $\delta$. Since $\theta$ is coarsely Lipschitz, this implies $\theta(\alpha)$ is uniformly close to $\delta$ for all $\alpha \in P(\mu)$. This shows $\Theta(\mu)$ is uniformly bounded in this case.
    \medskip

      \emph{Case 2: $g(W) =0$,     
      and none of the components of $S \setminus W$  are in $\fS$.}   This implies that $\alpha'$ is a vertex of $P(\mu)$. Moreover, $\alpha$ can be connected to $ \alpha'$ with at most $\xi$ edges of $\cK$ (one for each curve removed to go from $\alpha$ to $\alpha'$).
      
      Since $W \in \fS$ but not in $\overline{\fS}$, there exists a multicurve in $\mathcal{D}$ that is disjoint from $W$. Each component of $S \setminus W$ has genus zero by Lemma \ref{lem:admwitnesses}. Thus, we can apply Lemma \ref{lem:lots_of_2-level_curves} to find some $\delta \in \mathcal{D}$ and $N>0$ depending only on $S$ such that $i(\mu, \delta) \leq N$ and $\delta$ is disjoint from $W$; see Figure \ref{fig:delta_and_m}. Note that this choice depends only on $\mu$, not on $\alpha \in P(\mu)$.

   Since $\delta$ is a 2-level splitting, $S \setminus \delta $ has a component $Z \subset Y_0(\delta)$ with $g(Z) \geq 1$. Since $Z$ contains an admissible curve and $W \in \fS$, we must have $W \subseteq Z$. 
    By Lemma \ref{lem:bdd int Mod orbit}, there is a uniform $N'>0$ and a (possibly empty) multicurve $m$ on $Z \setminus W$ such that $m$ cuts $Z \setminus W$ into three-holed spheres and $i(m,\mu) \leq N'$; see Figure \ref{fig:delta_and_m}. Since $m$ cuts $Z \setminus W$ into three-holed spheres,  $\alpha_W \cup m \cup \partial W \cup \delta$ is a vertex of $\cK$ that intersects $\alpha'$ at most $N+N'$ times; see Figure \ref{fig:delta_and_m} for a schematic of the situation. Let $\delta' =  m \cup \partial W \cup \delta$, so $i(\alpha', \alpha_W\cup \delta') \leq N+N'$.

    Thus $d_\cK(\alpha',\alpha_W \cup \delta')$ is bounded uniformly by some number determined by $N+N'$.
    \begin{figure}[h]
        \centering
        \def\svgwidth{0.9\linewidth}
        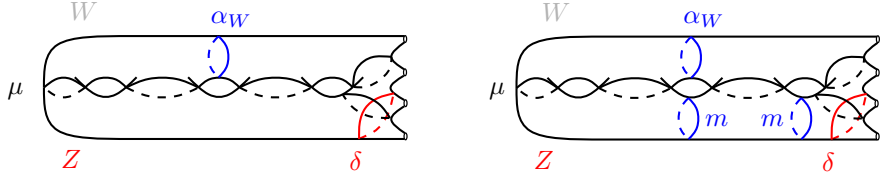
        \caption{A schematic of the curves $\delta$ and $m$ relative to $\mu$ and $W$. The actual intersection number between $\delta \cup m$ and $\mu$ is possibly higher, but still uniformly bounded.}
        \label{fig:delta_and_m}
    \end{figure}

    As in the previous case, Lemma \ref{lem:sep_curve_int_multicurve} says $\Psi(\alpha_W \cup \delta')$ will contain a curve $c \subset Z$ that cuts off a genus 1 subsurface of $Z$.
    Hence $\Psi(\alpha_W \cup \delta')$ will contain an admissible curve that is disjoint from $\delta$. This means  $\theta(\alpha_W \cup \delta')$ is a bounded diameter set that contains a vertex that is adjacent to $\delta$ in $\mathscr{E}$.
Thus, since $\alpha$ is uniformly close to $\alpha'$ which is in turn uniformly close to $\alpha_W \cup \delta'$ and $\theta$ is coarsely Lipschitz, we conclude that $\theta(\alpha)$ is uniformly close to $\delta$.
Since $\delta$ depended only on $\mu$, this implies $\diam(\Theta(\mu))$ is uniformly bounded.
    \medskip

    \emph{Case 3: $g(W) =0$ and $S \setminus W$ has a component $V$ that is in $\fS$.} By  Lemma \ref{lem:admwitnesses}, there is only one such component $V$. Let $\beta$ be a second vertex of $P(\mu)$ alongside $\alpha$.
    Recall $\alpha_W = \alpha \cap W$ and similarly define $\alpha_V$, $\beta_W$, and $\beta_V$. 
    Note that $\mu \cup \alpha_W$ and $\mu \cup \beta_V$ are vertices of $\overline{\cK}$ that we have already shown in the previous cases to have bounded diameter image under $\Theta$.
    Now observe 
\[\alpha \in P(\mu \cup \alpha_W), \,
\alpha_W \cup \mu \cup \beta_V \in P(\mu \cup \alpha_W) \cap P(\mu \cup \beta_v),
\text{ and } \beta \in P(\mu \cup \beta_V).\]
Hence $\theta(\alpha)$ is uniformly close to $\theta(\alpha_W \cup \mu \cup \beta_V)$, which is in turn uniformly close to $ \theta(\beta_W \cup \mu \cup \beta_V)$. Thus $\Theta(\mu)$ is uniformly bounded.

\medskip
{\noindent \bf Step 2: edges have uniform diameter.}
If $\nu$ is obtained from $\mu$ by adding a curve, then $P(\nu)\subseteq P(\mu)$.
Hence $\Theta(\mu) \cup \Theta(\nu) = \Theta(\mu)$ which has bounded diameter by Step 1.

Otherwise, the edge from $\mu$ to $\nu$ corresponds to a flip move.
Let $x \in \mu$ and $x'\in \nu$ be such that $x$ is flipped to $x'$ and let $Y$ be the component of $S \setminus (\mu \setminus x)$ containing $x$ and $x'$.
If $\xi(Y) >1$, then  $i(x,y) = 0$ and $\mu \cup x'$ is a vertex of $\overline{\cK}$.
Now $\mu \cup x'$ is joined by a ``remove'' edge to both $\mu$ and $\nu$ (removing $x'$ gives $\mu$ and removing $x$ gives $\nu$), so the diameter bound follows from the add/remove edge case.
If $\xi(Y) = 1$, then $x$ and $x'$ intersect minimally on $Y$. We can therefore find two pants decompositions $\alpha \in P(\mu)$ and $\alpha' \in P(\nu)$ such that $\alpha$ differs from $\alpha'$ by flipping $x$ to  $x'$.
Since $\alpha$ and $\alpha'$ are joined by an edge in $\cK$, the sets $\theta(\alpha)$ and   $\theta(\alpha')$ are uniformly close in $\mathscr{E}$. Since $\alpha \in P(\mu)$ and $\alpha' \in P(\nu)$, this implies that $\Theta(\mu) \cup \Theta(\nu)$ has uniformly bounded diameter.

\medskip
{\noindent \bf Step 3: $\Theta$ is a coarse inverse of the inclusion.} Let $j \colon \mathscr{E} \to \overline{\cK}$ be the 2-Lipschitz inclusion map.
Let $\mu \in j(\mathscr{E})$, that is, $\mu$ is either an admissible curve or $\mu \in \mathcal{D}$.
If $\mu$ is admissible, then $\mu$ is a vertex of $P(\mu)$ and $\Pi \circ \Psi (\mu)$ contains the admissible curve $\mu$ (as in the proof of Proposition \ref{prop:comp_coarseLip}), so $\mu \in \Theta(\mu)$.
If $\mu \in \mathcal{D}$, then there is some admissible curve $a$ contained in $Y_0(\mu)$ that is in particular disjoint from $\mu$.
Since $a \cup \mu$ and $\mu$ are joined by an edge of $\overline{\cK}$, 
this means $\Theta(a\cup \mu)$ and $\Theta(\mu)$ are uniformly close in $\mathscr{E}$ by Step 2. Now $a \in \Theta(a \cup \mu)$ because it is an admissible curve.
Thus $\Theta(\mu)$ is uniformly close to $a$, which is joined to $\mu$ by an edge in $\CRV$. 
\end{proof}

\begin{proof}[Proof of Theorem \ref{mainthm:bdrycx}]
By Corollary \ref{cor:variants_qi}, $\CRV$ is quasi-isometric to $\CRE$ and $\mathscr{C}(\barmarkH)$, and by Theorem \ref{thm:bdrycx model qi}, it is quasi-isometric to the hierarchically hyperbolic model $\overline \cK$.
Thus all of these graphs are hierarchically hyperbolic.

The proof that $\mathscr{C}(\barmarkH)$ has a pair of disjoint witnesses is a slight variation of the one appearing in the proof of Theorem \ref{mainthm:CadmHHS}.
Since the restriction of $\phi$ to the top level of any 2-level splitting must be of holomorphic type, the winding numbers of each curve of $\beta$ must be negative, and by homological coherence are all bounded by $\chi(S)$.

Now consider a multicurve $\alpha$ separating $S$ into a pair of genus $0$ subsurfaces $W^{\pm}$ such that $W^+$ contains no punctures and lies to the left of each curve of $\alpha$.
In order for $W^+$ to contain either an admissible curve or a curve of a 2-level splitting then there must be some subset $\mathcal C$ of the curves of $\alpha$ such that
\[1 - |\mathcal C| \le \sum_{c\in \mathcal C} \phi(c) \le 1-  |\mathcal C| - \chi(S).\]
Thus, by choosing an $\alpha$ with large enough winding numbers such that no subset sums of the winding numbers of its curves are in this small range, we see that there can be no admissible curves {\em or} 2-level splittings contained in $W^+$ and hence $W^-$ is a witness. A similar argument implies that for sufficiently large choices of the winding numbers of $\alpha$, the subsurface $W^+$ will also be a witness, and so $\ocK$ cannot be Gromov hyperbolic.
\end{proof}

\bibliographystyle{amsalpha}

\bibliography{bib}

\newcommand{\etalchar}[1]{$^{#1}$}
\providecommand{\bysame}{\leavevmode\hbox to3em{\hrulefill}\thinspace}
\providecommand{\MR}{\relax\ifhmode\unskip\space\fi MR }
\providecommand{\MRhref}[2]{%
  \href{http://www.ams.org/mathscinet-getitem?mr=#1}{#2}
}
\providecommand{\href}[2]{#2}
\begin{thebibliography}{BCG{\etalchar{+}}19}

\bibitem[ABW23]{ABW}
Paul Apisa, Matt Bainbridge, and Jane Wang, \emph{Moduli spaces of complex
  affine and dilation surfaces}, J. Reine Angew. Math. \textbf{796} (2023),
  229--243.

\bibitem[AHW24]{AHW_geo}
F.~Arana-Herrera and A.~Wright, \emph{The geometry of totally geodesic
  subvarieties of moduli spaces of riemann surfaces}, Preprint,
  \href{http://arxiv.org/abs/2412.16330}{arXiv:2412.16330}, 2024.

\bibitem[BCG{\etalchar{+}}18]{IVC}
M.~Bainbridge, D.~Chen, Q.~Gendron, S.~Grushevsky, and M.~M\"{o}ller,
  \emph{Compactification of strata of {A}belian differentials}, Duke Math. J.
  \textbf{167} (2018), no.~12, 2347--2416.

\bibitem[BCG{\etalchar{+}}19]{multiscale}
\bysame, \emph{The moduli space of multi-scale differentials}, Preprint,
  \href{http://arxiv.org/abs/1910.13492}{arXiv:1910.13492}, 2019.

\bibitem[BCM12]{BCM_ELC2}
J.~F. Brock, R.~D. Canary, and Y.~N. Minsky, \emph{The classification of
  {K}leinian surface groups, {II}: {T}he ending lamination conjecture}, Ann. of
  Math. (2) \textbf{176} (2012), no.~1, 1--149. \MR{2925381}

\bibitem[BF06]{BF_WP_rank}
J.~Brock and B.~Farb, \emph{Curvature and rank of {T}eichm\"{u}ller space},
  Amer. J. Math. \textbf{128} (2006), no.~1, 1--22. \MR{2197066}

\bibitem[BHS17a]{BHS_HHS_AsDim}
J.~Behrstock, M.~F. Hagen, and A.~Sisto, \emph{Asymptotic dimension and
  small-cancellation for hierarchically hyperbolic spaces and groups}, Proc.
  Lond. Math. Soc. (3) \textbf{114} (2017), no.~5, 890--926. \MR{3653249}

\bibitem[BHS17b]{BHS_HHSI}
\bysame, \emph{Hierarchically hyperbolic spaces {I}: {C}urve complexes for
  cubical groups}, Geom. Topol. \textbf{21} (2017), no.~3, 1731--1804.
  \MR{3650081}

\bibitem[BHS19]{BHS_HHSII}
\bysame, \emph{{Hierarchically hyperbolic spaces {II}: {C}ombination theorems
  and the distance formula}}, Pacific J. Math. \textbf{299} (2019), 257--338.

\bibitem[BM04]{BM_johnson_kernel}
T.~E. Brendle and D.~Margalit, \emph{Commensurations of the {J}ohnson kernel},
  Geom. Topol. \textbf{8} (2004), 1361--1384. \MR{2119299}

\bibitem[Boi15]{Boissy}
Corentin Boissy, \emph{Connected components of the strata of the moduli space
  of meromorphic differentials}, Comment. Math. Helv. \textbf{90} (2015),
  no.~2, 255--286.

\bibitem[Bro03]{Brock_WP_volumes}
J.~F. Brock, \emph{The {W}eil-{P}etersson metric and volumes of 3-dimensional
  hyperbolic convex cores}, J. Amer. Math. Soc. \textbf{16} (2003), no.~3,
  495--535. \MR{1969203}

\bibitem[Cal20]{strata1}
A.~Calderon, \emph{Connected components of strata of {A}belian differentials
  over {T}eichm\"{u}ller space}, Comment. Math. Helv. \textbf{95} (2020),
  no.~2, 361--420. \MR{4115287}

\bibitem[CF22]{ChenFaraco}
D.~Chen and G.~Faraco, \emph{Period realization of meromorphic differentials
  with prescribed invariants}, Preprint,
  \href{http://arxiv.org/abs/2212.05754}{arXiv:2212.05754}, 2022.

\bibitem[CMSZ20]{CMSZ_int}
D.~Chen, M.~M\"{o}ller, A.~Sauvaget, and D.~Zagier, \emph{Masur-{V}eech volumes
  and intersection theory on moduli spaces of {A}belian differentials}, Invent.
  Math. \textbf{222} (2020), no.~1, 283--373.

\bibitem[CMZ22]{CMZEuler}
M.~Costantini, M.~M\"{o}ller, and J.~Zachhuber, \emph{The {C}hern classes and
  {E}uler characteristic of the moduli spaces of {A}belian differentials},
  Forum Math. Pi \textbf{10} (2022), Paper No. e16, 55.

\bibitem[CS22]{CS}
A.~Calderon and N.~Salter, \emph{Framed mapping class groups and the monodromy
  of strata of abelian differentials}, J. Eur. Math. Soc. (2022), (to appear).

\bibitem[EMZ03]{EMZ}
A.~Eskin, H.~Masur, and A.~Zorich, \emph{Moduli spaces of abelian
  differentials: the principal boundary, counting problems, and the
  {S}iegel-{V}eech constants}, Publ. Math. Inst. Hautes \'{E}tudes Sci. (2003),
  no.~97, 61--179.

\bibitem[FI05]{FI_Torelli}
B.~Farb and N.~V. Ivanov, \emph{The {T}orelli geometry and its applications:
  research announcement}, Math. Res. Lett. \textbf{12} (2005), no.~2-3,
  293--301. \MR{2150885}

\bibitem[FM12]{FarbMarg}
B.~Farb and D.~Margalit, \emph{A primer on mapping class groups}, Princeton
  Mathematical Series, vol.~49, Princeton University Press, Princeton, NJ,
  2012.

\bibitem[Har81]{Harvey_CS}
W.~J. Harvey, \emph{Boundary structure of the modular group}, pp.~245--252,
  Princeton University Press, 1981.

\bibitem[Hen20]{Hensel_handlebody}
S.~Hensel, \emph{A primer on handlebody groups}, Handbook of group actions.
  {V}, Adv. Lect. Math. (ALM), vol.~48, Int. Press, Somerville, MA, [2020]
  \copyright 2020, pp.~143--177. \MR{4237892}

\bibitem[HJ89]{HJ}
S.~Humphries and D.~Johnson, \emph{{A generalization of winding number
  functions on surfaces}}, Proc. London Math. Soc. \textbf{58} (1989), no.~2,
  366--386.

\bibitem[HK14]{HK_DM}
J.~H. Hubbard and S.~Koch, \emph{An analytic construction of the
  {D}eligne-{M}umford compactification of the moduli space of curves}, J.
  Differential Geom. \textbf{98} (2014), no.~2, 261--313.

\bibitem[Iva97]{Ivanov_curve_graph}
N.~V. Ivanov, \emph{Automorphisms of complexes of curves and of
  {T}eichm\"{u}ller spaces}, Progress in knot theory and related topics,
  Travaux en Cours, vol.~56, Hermann, Paris, 1997, pp.~113--120. \MR{1603146}

\bibitem[Kaw18]{Kawazumi}
N.~Kawazumi, \emph{The mapping class group orbits in the framings of compact
  surfaces}, Q. J. Math. \textbf{69} (2018), no.~4, 1287--1302.

\bibitem[KZ]{KZ_strings}
M.~Kontsevich and A.~Zorich, \emph{{Lyapunov exponents and Hodge theory}},
  Preprint, 1--16.

\bibitem[KZ03]{KZ}
\bysame, \emph{Connected components of the moduli spaces of {A}belian
  differentials with prescribed singularities}, Invent. Math. \textbf{153}
  (2003), no.~3, 631--678.

\bibitem[LM14]{LM}
E.~Looijenga and G.~Mondello, \emph{{The fine structure of the moduli space of
  abelian differentials in genus 3}}, Geom. Dedicata \textbf{169} (2014),
  no.~1, 109--128.

\bibitem[Mas86]{Masur_existscyl}
H.~Masur, \emph{Closed trajectories for quadratic differentials with an
  application to billiards}, Duke Math. J. \textbf{53} (1986), no.~2, 307--314.

\bibitem[Min10]{Minsky_ELC1}
Y.~N. Minsky, \emph{The classification of {K}leinian surface groups. {I}.
  {M}odels and bounds}, Ann. of Math. (2) \textbf{171} (2010), no.~1, 1--107.
  \MR{2630036}

\bibitem[MM98]{MM1}
H.~Masur and Y.~N. Minsky, \emph{{Geometry of the complex of curves I:
  Hyperbolicity}}, Invent. Math. \textbf{138} (1998), 103--149.

\bibitem[MM00]{MM2}
\bysame, \emph{Geometry of the complex of curves. {II}. {H}ierarchical
  structure}, Geom. Funct. Anal. \textbf{10} (2000), no.~4, 902--974.
  \MR{1791145}

\bibitem[MS13]{MS_disk_graph}
H.~Masur and S.~Schleimer, \emph{The geometry of the disk complex}, J. Amer.
  Math. Soc. \textbf{26} (2013), no.~1, 1--62. \MR{2983005}

\bibitem[MUW21]{MUW}
M.~M\"{o}ller, M.~Ulirsch, and A.~Werner, \emph{Realizability of tropical
  canonical divisors}, J. Eur. Math. Soc. \textbf{23} (2021), no.~1, 185--217.

\bibitem[Put08]{Putman_connectivity}
A.~Putman, \emph{A note on the connectivity of certain complexes associated to
  surfaces}, Enseign. Math. (2) \textbf{54} (2008), no.~3-4, 287--301.
  \MR{2478089}

\bibitem[Raf05]{Rafi_short_curves}
K.~Rafi, \emph{A characterization of short curves of a {T}eichm\"{u}ller
  geodesic}, Geom. Topol. \textbf{9} (2005), 179--202. \MR{2115672}

\bibitem[RS09]{RScovers}
Kasra Rafi and Saul Schleimer, \emph{Covers and the curve complex}, Geom.
  Topol. \textbf{13} (2009), no.~4, 2141--2162.

\bibitem[RV19]{russell_vokes_sep}
Jacob Russell and Kate~M. Vokes, \emph{The (non)-relative hyperbolicity of the
  separating curve graph}, arXiv:1910.01051 (2019).

\bibitem[RW14]{RW}
O.~Randal-Williams, \emph{Homology of the moduli spaces and mapping class
  groups of framed, {$r$}-{S}pin and {P}in surfaces}, J. Topol. \textbf{7}
  (2014), no.~1, 155--186.

\bibitem[Sal]{salter_higher_spin}
N.~Salter, \emph{Higher spin mapping class groups in algebraic and flat
  geometry}, Notes available at
  \url{https://nsalter.science.nd.edu/expository-notes/cuernavacalectures.pdf}.

\bibitem[Tan21]{Tang}
Robert Tang, \emph{Affine diffeomorphism groups are undistorted}, J. Lond.
  Math. Soc. (2) \textbf{104} (2021), no.~2, 747--769.

\bibitem[Vok22]{Vokes_HHS}
K.~Vokes, \emph{Hierarchical hyperbolicity of graphs of multicurves}, Algebr.
  Geom. Topol. \textbf{22} (2022), no.~1, 113--151.

\bibitem[Won24]{Wongspin}
Y.~M. Wong, \emph{An algorithm to compute the fundamental classes of spin
  components of strata of differentials}, Int. Math. Res. Not. IMRN (2024),
  no.~6, 4893--4962.

\end{thebibliography}
\end{document}